\documentclass[12pt,twoside,a4paper]{article}

\usepackage{amsmath,amssymb,amsfonts,amsthm}
\usepackage{wasysym}
\usepackage{psfrag}
\usepackage{graphicx}  
\usepackage{comment}
\usepackage{enumerate}
\usepackage{color}
\usepackage{appendix}
\usepackage{mathrsfs}

\usepackage{txfonts}

\theoremstyle{plain}
\newtheorem{theorem}{Theorem}
\newtheorem{lemma}[theorem]{Lemma}
\newtheorem{prop}[theorem]{Proposition}

\theoremstyle{definition}

\newtheorem{definition}[theorem]{Definition}

\DeclareMathOperator*{\argmax}{arg\,max}
\DeclareMathOperator*{\argmin}{arg\,min}

\begin{document}

\title{Convergence of discrete Aubry-Mather model in the continuous limit}

\author{Xifeng Su \footnote{School of Mathematical Sciences,
Beijing Normal University,
No. 19, XinJieKouWai St.,HaiDian District,
 Beijing 100875, P. R. China,
\texttt{xfsu@bnu.edu.cn}}
\and
Philippe Thieullen \footnote{Institut de Math\'ematiques de Bordeaux,
Universit\'e de Bordeaux,
351, cours de la Lib\'eration - F 33405 Talence, France,
\texttt{philippe.thieullen@u-bordeaux.fr}
}
} 

\date{\today}

\maketitle
\tableofcontents

\begin{abstract}
We develop two approximation schemes for solving the cell equation and the discounted cell equation  using Aubry-Mather-Fathi theory. The Hamiltonian is supposed to be Tonelli, time-independent , and periodic in space. By Legendre transform it is equivalent to find a fixed point of some nonlinear operator, called Lax-Oleinik operator, which may be discounted or not. By discretizing in time, we are led to solve an additive eigenvalue problem involving a discrete Lax-Oleinik operator. We show how to approximate the effective Hamiltonian and  some weak KAM solutions  by letting the time step in the discrete model tend to zero. We also obtain a selected discrete  weak KAM solution as in \cite{DaviniFathiIturriagaZavidovique2014} and show it converges to a particular solution of the cell equation. In order to unify the two settings, continuous and  discrete , we develop a more general formalism of short-range interactions.

{\bf Keywords: }discrete weak KAM theory, Frenkel-Kontorova models,  
Aubry-Mather theory, discounted Lax-Oleinik operator, ergodic cell equation, short-range interactions, additive eigenvalue problem
\end{abstract}

\section{Introduction}

We consider in  this article a Hamiltonian $H(x,p):\mathbb{T}^d\times \mathbb{R}^d\rightarrow \mathbb{R}$ which is $C^2$, periodic in $x$, time independent, and satisfies the following assumptions:
\begin{itemize}
\item [(L1)] \textbf{Positive Definiteness:} $H(x,p)$ is strictly convex with respect to $p$, i.e., the second partial derivative $\frac{\partial^2 H}{\partial p^2}(x,p)$ is positive definite as a quadratic form uniformly in $x \in \mathbb{T}^d$ and $\|p\| \leq R$, for every $R>0$;

\item [(L2)] \textbf{Superliner growth:} $H(x,p)$ is superlinear with respect  to  $p$, uniformly in $x$, that is,
\[
\lim_{\|p\|\rightarrow +\infty} \inf_{x\in\mathbb{T}^d} \frac{H(x,p)}{\|p\|}  = +\infty.
\]
\end{itemize}

\noindent We will say that $H(x,p)$ is  a {\it Tonelli Hamiltonian}. We denote by $L(x,v)$ the Legendre-Fenchel transform of $H(x,p)$. We call $L(x,v)$ the {\it Lagrangian} of the system;  $L(x,v)$  is again $C^2$, strictly convex with respect to $v$, and superlinear.  A more general framework could be chosen where $\mathbb{T}^d \times \mathbb{R}^d$ is replaced by the cotangent space $T^*M$ of some compact manifold $M$, but this approach would increase the complexity of the notations.  To illustrate the two approximation schemes we are going to present, we choose the following basic Hamiltonian
\[
H(x,p) = \frac{1}{2}\|p+P\|^2 -K(1-\cos(2\pi \, N \cdot x)),
\]
where $P \in \mathbb{R}^d, N \in \mathbb{Z}^d$ and $K \in \mathbb{R}$ are three parameters. The Lagrangian becomes
\[
L(x,v ) = \frac{1}{2}\|v\|^2  -P \cdot v  + K(1-\cos(2\pi \, N\cdot  x)).
\]

We consider the following two equations: the {\it PDE cell equation} and the {\it discounted PDE cell equation},
\begin{gather}
H(x,du(x))=\bar H,  \label{equation:PDEcellEquation} \\
\delta u_\delta(x) + H(x,d_xu_\delta(x)) = 0, \label{equation:DiscountedCellEquation}
\end{gather}
where $u(x)$ and $u_\delta(x)$ are understood in the viscosity sense. Our main objective is to describe an ergodic approximation scheme for each equation.

Equation \eqref{equation:PDEcellEquation} is a degenerate PDE equation of first order with two unknowns $(\bar H, u)$. The constant $\bar H$ is  unique and  is  called  {\it effective Hamiltonian}. The function $u(x)$ is  $C^0$ periodic but may not be unique. Equation \eqref{equation:DiscountedCellEquation} is more regular and admits a unique $C^0$ periodic solution $u_\delta(x)$. Equation \eqref{equation:PDEcellEquation} has first been studied by Lions, Papanicolaou and Varadhan \cite{Lions'87}. A comprehensive treatment may be found in Crandall, Ishii and Lions \cite{Crandall'92}, Bardi and Capuzzo-Dolcetta \cite{BardiDolcetta'97} or Barles \cite{Barles'94}. Some recent overviews may be found in the articles  \cite{Ishii2013,Barles2013}. 

A new approach has been initiated by Mather and Fathi \cite{Mather'91, Mather'93,Fathi'97_1, Fathi'97_2, Fathi'08} to solve equation \eqref{equation:PDEcellEquation}. Fathi showed that   \eqref{equation:PDEcellEquation}  is equivalent to  an additive eigenvalue problem for a semi-group of non-linear operators,
\begin{align}
u(x) - t \bar H &= T^t[u](x),\quad \forall t>0, \ \forall x\in\mathbb{R}^d, \label{equation:ErgodicCellEquation} \\
T^t|u](x) &:= \inf_{\substack{\gamma\in C^{ac}([-t,0], \mathbb{R}^d) \\ \gamma(0)=x}} \Big[ u(\gamma(-t)) + \int_{-t}^0\!  L(\gamma,\dot\gamma) \ ds \Big], \label{equation:LinearProgrammingPrinciple}
\end{align}
(where the infimum is taken over absolutely continuous paths over $[-t,0]$ with terminal point $x \in \mathbb{R}^d$). For Tonelli Hamiltonian, the  infimum is actually attained by a $C^2$ curve thanks to  Tonelli-Weierstrass theorem.   

Equation \eqref{equation:ErgodicCellEquation} is called the  {\it ergodic cell equation}, $T^t$ is called the {\it (backward) Lax-Oleinik semi-group}. The unknown $u(x)$ is called  by Fathi {\it weak KAM solution}, $\bar H$ is as before the effective Hamiltonian.  Ma\~n\'e \cite{Mane'96} recognized first the importance of this constant $\bar H$. After Contreras and Iturriaga \cite{Contreras'99}, $\bar H$ is called the  {\it Ma\~n\'e critical value}: $\bar H$ has the explicit value
\begin{equation} \label{equation:EffectiveHamiltonian}
-\bar H := \lim_{t \to +\infty} \  \inf_{\gamma\in C^{ac}([-t,0],\mathbb{R}^d)} \Big[ \frac{1}{t}  \int_{-t}^0 L(\gamma,\dot\gamma) \,ds \Big].
\end{equation}

Equation \eqref{equation:DiscountedCellEquation} has been studied by \cite{Lions'87,Crandall'92,Barles'94,BardiDolcetta'97}. The solution is unique and given explicitly  by the integral formula
\begin{equation}\label{equation:DiscoutedProgrammingPrinciple}
u_\delta(x) = \inf_{\substack{\gamma \in C^2((-\infty,0],\mathbb{R}^d) \\ \gamma(0)=x}} \int_{-\infty}^0\! e^{s\delta} L(\gamma(s),\dot\gamma(s)) \, ds,
\end{equation} 
where the infimum is taken over  $C^2$ paths ending at $x$ with a uniformly bounded first and second derivative. The two equations \eqref{equation:PDEcellEquation} and \eqref{equation:DiscountedCellEquation} are related, but very recently, the authors of \cite{DaviniFathiIturriagaZavidovique2014} showed that  $u_\delta(x)$, correctly normalized, converges to a selected solution $u^*(x)$ of \eqref{equation:ErgodicCellEquation}, 
\begin{equation}\label{equation:AsymptoticallyDiscountedSolution}
\lim_{\delta\to 0} \left( u_\delta(x) + \frac{\bar H}{\delta} \right) = u^*(x) \quad\text{(exists in the $C^0$ topology).}
\end{equation}
We will call this selected solution $u^*$, the {\it balanced  weak KAM solution}. 

Our main objective is to develop  approximation schemes that solves  \eqref{equation:PDEcellEquation} and \eqref{equation:DiscountedCellEquation}. In the first  scheme, we compute an approximated effective Hamiltonian of \eqref{equation:EffectiveHamiltonian} and an approximated weak KAM solution of \eqref{equation:ErgodicCellEquation}. In the second  scheme, we compute an approximated discounted weak KAM solution of  \eqref{equation:DiscoutedProgrammingPrinciple} and show a similar selection principle. In both cases we discretize in time, either the semi-group \eqref{equation:LinearProgrammingPrinciple} or the integral formula \eqref{equation:DiscoutedProgrammingPrinciple},  and rewrite the two problems in the framework of Frenkel-Kontorova model.

The  Frenkel-Kontorova model  has been studied in solid state physics in 1D by \cite{FK'38} and then more rigorously by Aubry and Le Daeron \cite{Aubry'83},  Chou and Griffiths \cite{Griffiths'86}, and in higher dimension by Gomes \cite{Gomes'05},  Garibaldi and Thieullen \cite{Thieullen'11}. Similar problems under the name of Aubry-Mather theory have been studied using transport theory by Bernard and Buffoni \cite{BernardBuffoni2007} and Zavidovique \cite{Zavidovique2012}. The Frenkel-Kontorova model describes the space of configurations of an infinite chain of atoms $(x_n)_{n\in\mathbb{Z}}$ at the ground-level energy. In this model  $x_n$ denotes the position of the n-th atom of the chain in $\mathbb{R}^d$, and $E(x_n,x_{n+1})$ denotes a {\it short-range interaction} between  two successive atoms. The interaction $E(x,y)$ models both the internal interaction between nearest atoms and  the external interaction with the substrate. The {\it original  Frenkel-Kontorova model} \cite{FK'38}  is given by 
\begin{equation*}
E(x,y) = \frac{1}{2} \|y-x\|^2  - P\cdot (y-x) + K(1-\cos(2\pi \, N \cdot  x)).
\end{equation*}
In solid state physics, it is more appropriate to write the elastic  interaction as $\frac{1}{2}\| y-x-P\|^2$ instead of $\frac{1}{2}\|y-x\|^2 - P\cdot(y-x)$ where $P$ denotes the mean distance at rest between two successive atoms of the chain. In Mather theory,  $P$ represents  a cohomological term. 

The main problem in the Frenkel-Kontorova model is to understand the set of configurations that minimize the total interaction $\sum_{n\in\mathbb{Z}} E(x_n,x_{n+1})$ in a precise sense. Chou and Griffiths \cite{Griffiths'86} highlighted first the importance of the two following quantities:   $\bar E$, the effective interaction  of the system  (or the ground-state energy in Gibbs theory), $u(x)$,  the effective potential which is a continuous periodic function that calibrates the interaction energy. They showed that $(\bar{E},u)$ can be seen as two unknowns of a discrete additive eigenvalue equation, now called,  {\it discrete (backward) Lax-Oleinik equation}, 
\begin{equation}\label{equation:DiscreteLaxOleinik}
u(y) + \bar E = \inf_{x\in\mathbb{R}^d} \big\{ u(x) + E(x,y) \big\}, \quad \forall y \in \mathbb{R}^d.
\end{equation}

The goal of the first scheme is to show that one can solve \eqref{equation:ErgodicCellEquation} by solving     \eqref{equation:DiscreteLaxOleinik} with the following interaction $E(x,y) = \mathcal{L}_{\tau}(x,y)$ and by letting $\tau \to 0$. We call  {\it discrete action},
\begin{equation} \label{equation:DiscreteAction}
\mathcal{L}_{\tau}(x,y) := \tau L\Big(x,\frac{y-x}{\tau}\Big), \quad \forall \tau>0.
\end{equation}
If $(\bar{\mathcal{L}}_\tau, u_\tau)$ is a solution of \eqref{equation:DiscreteLaxOleinik}, one obtains in particular
\[
\lim_{\tau\to0} \frac{\bar{\mathcal{L}}_\tau}{\tau} = -\bar H, \quad \lim_{\tau_i\rightarrow 0} u_{\tau_i} = u\  (\text{ for some subsequence } \tau_i\searrow 0).
\]
The discrete action associated to the basic example is given  for instance by
\[
E_\tau(x,y) = \frac{1}{2\tau} \|y-x\|^2  -P \cdot (y-x) +\tau K(1-\cos(2\pi \, \cdot x)).
\]
We recognize the original Frenkel-Kontorova model by taking $\tau=1$. Notice that \eqref{equation:ErgodicCellEquation} can trivially be written as a discrete Lax-Oleinik equation  with the following short-range interaction $E(x,y) = \mathcal{E}_\tau(x,y)$.  We call {\it minimal action}
\begin{equation}\label{equation:MinimalAction}
\mathcal{E}_{ \tau} (x,y) := \inf_{\substack{\gamma\in C^{ac}([0,\tau], \mathbb{R}^d)\\ \gamma(0)=x ,\ \gamma(\tau)=y} }  \int_0^\tau L( \gamma(t),\dot{\gamma}(t)) \ dt, \quad \forall \tau>0, \ \forall x,y \in\mathbb{R}^d.
\end{equation}
The infimun can be realized by some $C^2$ curve thanks to Tonelli--Weierstrass  theorem. We will use  $\mathcal{L}_\tau(x,y)$ as  a numerical tool to solve  \eqref{equation:ErgodicCellEquation}. Several algorithms can be used to solve  \eqref{equation:DiscreteLaxOleinik} like Ishikawa's iterative method. We will use  $\mathcal{E}_\tau(x,y)$  as a theoretical tool to prove the convergence of the scheme.  

The goal of the second scheme is to extend in the discrete case the main result of Davini, Fathi, Iturriaga, and Zavidovique in their first paper \cite{DaviniFathiIturriagaZavidovique2014}. We were aware of a second paper  \cite{DFIZ2} related to ours after this paper was completed. However, in the latter paper, the authors do not consider the convergence issues of the approximations scheme. We will show in particular that the solution $u_{\tau,\delta}$ of the {\it discounted discrete Lax-Oleinik equation}
\begin{equation} \label{equation:DIscountedDiscreteLaxOleinik}
u_{\tau,\delta}(y ) = \inf_{x \in \mathbb{R}^d} \big\{ (1-\tau\delta) u_{\tau,\delta}(x) + \mathcal{L}_\tau(x,y \big\}, \forall \tau>0, \ \forall y \in \mathbb{R}^d
\end{equation}
satisfies for every $\tau>0$, $\displaystyle{\lim_{\delta \to 0} \Big( u_{\tau,\delta}-\frac{\bar{\mathcal{L}}_\tau}{\tau\delta} \Big) = u_\tau^*}$ \  and \  $\displaystyle{\lim_{\substack{\tau,\delta \to 0 \\ \tau/\delta \to 0}} \Big( u_{\tau,\delta} -\frac{\bar{\mathcal{L}}_\tau}{\tau\delta} \Big) =  u^*}$.

\section{Main results}

The two previous short-range interactions ${\mathcal{L}_\tau}(x,y)$ and  ${\mathcal{E}_\tau}(x,y)$ belong to a class of parametrized interactions that we are going to discuss. We focus on the following definition on the fact that $\|y-x\|$, (the sup norm),  and $\tau$  should have the same order of magnitude as $\tau\to0$: we call this property {\it short-range}.

\begin{definition}\label{MainHypothesis}
We call {\it  short-range interaction}, a one-parameter family of functions $E_\tau(x,y) : \mathbb{R}^d \times \mathbb{R}^d \to \mathbb{R}$ indexed by $\tau>0$ satisfying:
\begin{itemize}
\item[(H1)] $E_\tau(x,y)$ is {\bf continuous} in $(x,y)$ for every $\tau>0$;
\item[(H2)] $E_\tau(x,y)$ is {\bf translational periodic} for every $\tau>0$:
 \[
E_{\tau} (x+k, y+k) = E_{\tau} (x,y), \quad \forall k\in\mathbb{Z}^d \quad \text{and} \quad \forall x,y\in\mathbb{R}^d;
\]
\item[(H3)] $E_\tau(x,y)$ is {\bf  coercive} for every $\tau>0$:
\[
\lim_{R\rightarrow +\infty}  \ \inf_{\|x-y\| \geq  R} \ E_{\tau}(x,y) = +\infty;
\]
\item[(H4)] $E_\tau(x,y)$ is {\bf uniformly bounded}: for every $R>0$
\begin{equation*}
\inf_{\tau \in (0,1]} \ \inf_{x,y \in \mathbb{R}^d} \ {\textstyle \frac{1}{\tau}}E_\tau(x,y) > -\infty, \quad
\sup_{\tau \in (0,1]} \ \sup_{\|y-x\| \leq \tau R} {\textstyle \frac{1}{\tau}} E_\tau(x,y) < +\infty;
\end{equation*}
\item[(H5)] $E_\tau(x,y)$ is {\bf uniformly superlinear}: 
\[
\lim_{R\rightarrow +\infty} \ \inf_{\tau\in (0,1]} \  \inf_{\|x-y\| \geq \tau R}  \frac{E_{\tau}(x,y)}{\|x-y\|}= +\infty;
\]
\item[(H6)] $E_\tau(x,y)$ is {\bf uniformly Lipschitz}: for every $R>0$, there exists a constant $C(R)>0$ such that, for every $\tau \in (0,1]$ and for every  $x,y,z \in \mathbb{R}^d$, 

\medskip
-- if $\|y-x\| \leq \tau R$ and $\|z-x\|   \leq \tau R$ then
\begin{equation*}
| E_\tau(x,z) - E_\tau(x,y) | \leq C(R) \|z-y\|,
\end{equation*}
-- if $\|z-x\| \leq \tau R$ and $\|z-y\| \leq \tau R$ then
\begin{equation*}
| E_\tau(x,z) - E_\tau(y,z) | \leq C(R) \|y-x\|.
\end{equation*}
\end{itemize}
We call {\it periodic  interaction} associated  to $E_\tau(x,y)$, the doubly periodic function
\[
E^*_\tau(x,y) := \inf_{k\in\mathbb{Z}^d} E_\tau(x,y+k).
\]
\end{definition}

The following proposition says that the two short-range interactions $\mathcal{L}_\tau(x,y)$ and $\mathcal{E}_\tau(x,y)$  are comparable in the sense that  $|\mathcal{L}_\tau(x,y) - \mathcal{E}_\tau(x,y)| = O(\tau^2)$ uniformly on  $\|y-x\| = O(\tau)$.

\begin{prop}[Comparison estimate] 
\label{proposition:ActionComparison}
Let $H:\mathbb{T}^d \times \mathbb{R}^d \to \mathbb{R}$ be a Tonelli Hamiltonian  and $L$ be the associated Lagrangian. 
\begin{enumerate}
\item \label{item:HypothesisVerification} The two short-range interactions $(\mathcal{L}_\tau(x,y))_{\tau>0}$   and $(\mathcal{E}_\tau(x,y))_{\tau>0}$, defined in  \eqref{equation:DiscreteAction} and  \eqref{equation:MinimalAction} respectively, satisfy the hypotheses (H1)--(H6).
\item \label{item:ComparaisonVerification}  For every $R>0$, there exists a constant   $C(R)>0$ such that, if $\tau \in (0,1]$, $x,y \in \mathbb{R}^d$ satisfy $\|y-x\| \leq \tau R$, then
\begin{equation*}
|\mathcal{E}_\tau(x,y) -\mathcal{L}_\tau(x,y)| \leq \tau^2 C(R).
\end{equation*} 
\end{enumerate}
\end{prop}

We recall two important definitions associated to an interaction:  the discrete Lax-Oleinik operator, and the  discrete weak KAM solution. The vocabulary is chosen so that it coincides to the new terminology used by Fathi in the case of  continuous time Lax-Oleinik operator.

\begin{definition} \label{definition:LaxOleinikOperator}
Let $(E_\tau(x,y))_{\tau>0}$ be a  short-range interaction satisfying (H1)--(H3).
\begin{itemize}
\item We call {\it discrete (backward) Lax-Oleinik operator}, 
\begin{gather*}\label{equation:ForwardBackwardLaxOleinik}
T_\tau [u] (y) := \min_{x \in \mathbb{R}^d} \big\{ u(x) + E_{\tau} (x, y) \big\},\ \quad \forall y\in\mathbb{R}^d,
\end{gather*}
acting on continuous  periodic functions $u \in C^0(\mathbb{R}^d)$. 
\item We call {\it discrete (backward) weak KAM solution for $E_\tau(x,y)$}, any periodic continuous function $u_\tau$ solution of the additive  eigenvalue problem,
\begin{equation}
\label{equation:AdditiveEigenvalueEquation}
T_\tau [u_\tau]  = u_\tau + \bar{E}_\tau, 
\end{equation}
for some $\bar E_\tau \in \mathbb{R}$.
\end{itemize}
Note that $T_\tau$ has the same definition if $E_\tau(x,y)$ is replaced by $E^*_\tau(x,y)$.
\end{definition}

We have defined two Lax-Oleinik operators: the first one  in the continuous case $T^t$ in \eqref{equation:LinearProgrammingPrinciple},  using a superscript $t$, the second one  in the discrete case $T_\tau$  in \eqref{equation:ForwardBackwardLaxOleinik}  using a subscript $\tau$. For the minimal action $\mathcal{E}_\tau(x,y)$ we have obviously $T^\tau = T_\tau$. 

We recall  a classical result on the existence of discrete weak KAM solutions for the Lax-Oleinik operator.  Different  proofs may be found as for instance in \cite{Nussbaum1991}, \cite{Gomes'05} or \cite{Thieullen'11}. 

\begin{prop}[Lax-Oleinik equation for short-range interactions]\label{theorem:AdditiveEigenvalueProblem}
We consider a short-range interaction $(E_\tau(x,y))_{\tau>0}$  satisfying the hypotheses (H1)--(H3). 
\begin{enumerate}
\item For every  $\tau > 0$, there exists a unique scalar $\bar{E}_\tau$  such that the equation  $T_\tau [u_\tau]  = u_\tau + \bar{E}_\tau$ admits  a continuous periodic solution $u_\tau$.
\item $\bar E_\tau$ is called effective interaction and can be computed in many ways
\begin{align}
\label{equation:GroundState}
\begin{split}
\bar{E}_\tau &= \sup_{u\in C^0(\mathbb{T}^d)} \ \inf_{x,y \, \in\, \mathbb{R}^d} \ \big\{ E_\tau(x,y) - [u(y)-u(x)] \big\}, \\
&= \sup_{v\in \mathcal{B}(\mathbb{R}^d)} \ \inf_{x,y \, \in\, \mathbb{R}^d} \ \big\{ E_\tau(x,y) - [v(y)-v(x)] \big\}, \\
&= {\textstyle \lim_{k\rightarrow +\infty} \ \inf_{z_0, \ldots, z_k \, \in\, \mathbb{R}^d} \  \frac{1}{k} \sum_{i=0}^{k-1} E_{\tau}(z_i,z_{i+1})}.
\end{split}
\end{align}
\end{enumerate}
$\mathcal{B}(\mathbb{R}^d)$ denotes the space of bounded  functions not necessarily periodic. Note that we could have used $E^*_\tau(x,y)$ instead of $E_\tau(x,y)$ in one of these formulas.
\end{prop}
The two first formulas are called the {\it sup-inf formula} and are analogue to the sup-inf formula  introduced by \cite{Contreras'98} for continuous-time Tonelli Hamiltonian systems. The third formula is called the {\it mean interaction per site formula}. Another characterization will be given in lemma~\ref{lemma:ErgodicEffectiveInteraction}.

The conclusions of proposition  \ref{theorem:AdditiveEigenvalueProblem} hold for both the discrete and the minimal action. There is no reason a priori that the two effective interactions $\bar{\mathcal{L}}_\tau$ and $\bar{\mathcal{E}}_\tau$ are comparable. The mean interaction per site formula suggests to consider minimizing paths $(z_0,\cdots,z_k)$. The following proposition shows that the jumps $\|z_k-z_{k-1}\|$ of such minimizing paths are uniformly comparable to $\tau$. We will be able to apply the proposition \ref{proposition:ActionComparison} and obtain $| \bar{\mathcal{L}}_\tau - \bar{\mathcal{E}}_\tau| = O(\tau^2)$.

\begin{prop}[A priori compactness for short-range interactions] \label{proposition:AprioriCompactnessLaxOleinik}
We consider a  short-range interaction  $(E_\tau(x,y))_{\tau>0}$ satisfying the hypotheses (H1)--(H6). 
\begin{enumerate}
\item There exist constants $C,R>0$ such that, if $\tau \in (0,1]$ and $u_\tau$ is a discrete weak KAM solution of $E_\tau(x,y)$, then
\begin{enumerate}
\item\label{item:AprioriLipschitzBound}  $u_\tau$ is  Lipschitz and $\text{\rm Lip}(u_\tau) \leq C$,
\item\label{item:AprioriJumpBound} $\forall y \in \mathbb{R}^d$,   $  \ x  \in \argmin_{x \in \mathbb{R}^d} \big\{u_\tau(x) + E_\tau(x,y) \big\} \  \ \Rightarrow \ \ \|y-x\| \leq \tau R$. 
\end{enumerate}

\item \label{item:AprioriLaxOleinikBound} For every Lipschitz periodic function $u$, $\lim_{\tau\to0} T_\tau[u] = u$ uniformly. More precisely, for every constant $\kappa>0$, there exist constants $R_\kappa, C_\kappa >0$ such that, if $u$ is any Lipschitz function satisfying $\text{Lip}(u) \leq \kappa$, and   $\tau \in (0,1]$, then
\begin{enumerate}
\item $\forall y \in \mathbb{R}^d$, $ \ x \in \argmin_{x\in\mathbb{R}^d} \big\{ u(x) + E_\tau(x,y) \big\} \  \ \Rightarrow \ \ \|y-x \| \leq \tau R_\kappa$, 
\item $\| \, T_\tau[u]-u \, \|_\infty \leq \tau C_\kappa$.
\end{enumerate}
\end{enumerate}
\end{prop}
Notice that the effective Hamiltonian \eqref{equation:EffectiveHamiltonian} can be written in the terminology of short-range interactions using the minimal action,
\begin{gather*} \label{equation:ConvolutionInteraction}
- \bar H = \lim_{\tau \to +\infty} \frac{1}{\tau} \min_{x,y \in \mathbb{R}^d} \mathcal{E}_\tau(x,y).
\end{gather*}
We show more generally how to solve equation \eqref{equation:ErgodicCellEquation} and how to obtain formula \eqref{equation:EffectiveHamiltonian} for any short-range interaction which is a min-plus convolution semi-group. 

\begin{definition} \quad
\begin{itemize}
\item We call {\it min-plus convolution} of two interactions $\mathcal{E}_1$ and $\mathcal{E}_2$, the interaction
\[
\mathcal{E}_1 \otimes \mathcal{E}_2 (x,y) := \inf_{z \in \mathbb{R}^d} [ \mathcal{E}_1(x,z) + \mathcal{E}_2(z,y) ], \quad \forall x,y \in \mathbb{R}^d.
\]
\item A short-range interaction $(E_\tau(x,y))_{\tau>0}$ is said to be a {\it min-plus convolution semi-group} if
\[
E_{\tau+\sigma} = E_\tau \otimes E_\sigma, \quad \forall \tau,\sigma >0.
\]
\end{itemize}
\end{definition}

The following observation is trivial and will not be proved.

\begin{lemma}
Let $H$ be a Tonelli Hamiltonian. Then the minimal action $(\mathcal{E}_\tau(x,y))_{\tau>0}$ is a min-plus convolution semi-group.
\end{lemma}

The following proposition extends \eqref{equation:ErgodicCellEquation} and \eqref{equation:EffectiveHamiltonian} for any short-range interaction which is a min-plus convolution semi-group. The proposition states there exists  a common additive eigenfunction  associated to a unique linear eigenvalue. 

\begin{prop}\label{theorem:FathiWeakKamConvolutionTheorem}
Let $(E_\tau(x,y))_{\tau > 0}$ be a short-range interaction satisfying (H1)--(H6). Assume  the interaction is a min-plus convolution semi-group.  Consider the equation
\begin{equation} \label{equation:FathiWeakKamConvolutionTheorem}
T_\tau[u] = u + \tau \bar{E}_1, \quad \forall \tau>0,
\end{equation}
where $u$ is a $C^0$ periodic function (independent of $\tau$) and $\bar E_1 \in \mathbb{R}$.
\begin{enumerate}
\item\label{item1:FathiWeakKamConvolutionTheorem} There exists a Lipschitz periodic function $u$ solution of \eqref{equation:FathiWeakKamConvolutionTheorem}.  Moreover
\[
\bar E_\tau = \tau \bar E_1, \quad \forall \tau>0.
\]
\item \label{item2:FathiWeakKamConvolutionTheorem} Let $u_\tau$ be any discrete weak KAM solution of $E_\tau(x,y)$. Assume $u_{\tau_i} \to u$ uniformly along a subsequence $\tau_i\to0$. Then $u$ is a Lipschitz solution of \eqref{equation:FathiWeakKamConvolutionTheorem}.
\item\label{item3:FathiWeakKamConvolutionTheorem} $\displaystyle{\lim_{\tau \to+\infty} \frac{1}{\tau} \min_{x,y\in\mathbb{R}^d}E_\tau(x,y) = \bar{E}_1}$. 
\end{enumerate}
\end{prop}

We summarize in the following theorem the previous results we have obtained for any short-range interactions to the particular case of discrete and minimal actions. We show how the PDE cell equation \eqref{equation:PDEcellEquation} can be approximated by a discrete weak KAM solution $u_\tau$. The speed of convergence to the effective Hamiltonian $\bar H$ is of the order $O(\tau)$. The convergence to the viscosity solution $u$ is obtained by taking a subsequence as $\tau \to 0$.

\begin{theorem}[First approximation scheme]\label{theorem:Main}
Let $H(x,p):\mathbb{T}^d\times\mathbb{R}^d \to \mathbb{R}$ be a Tonelli  Hamiltonian and $L(x,v)$ be the associated Lagrangian. We consider the two equations
\begin{align}
&u_\tau(y) + \bar{\mathcal{L}}_\tau = \min_{x\in\mathbb{R}^d} \big\{ u_\tau(x) + \mathcal{L}_\tau(x,y) \big\}, \quad \forall y \in \mathbb{R}^d, \ \forall \tau>0, \tag{E1}\label{equation:DiscreteTimeLaxOleinikEquation} \\
&u(y)  - \tau \bar H = \min_{x\in\mathbb{R}^d} \big\{ u(x) + \mathcal{E}_\tau(x,y) \big\}, \quad \forall y \in \mathbb{R}^d, \ \forall \tau>0, \tag{E2} \label{equation:ContinuousTimeLaxOleinkEquation}
\end{align}
where $u_\tau,u$ are $C^0$ periodic functions. 
\begin{enumerate}
\item\label{item1:MainTheorem} There is a unique $\bar{\mathcal{L}}_\tau$ such that  \eqref{equation:DiscreteTimeLaxOleinikEquation} admits a solution $u_\tau$. Moreover
\begin{gather*}
\label{equation:GroundState}
\bar{\mathcal{L}}_\tau  = \lim_{k\rightarrow +\infty} \ \inf_{z_0, \ldots, z_k \, \in\, \mathbb{R}^d} \  \frac{1}{k} \sum_{i=0}^{k-1} \mathcal{L}_{\tau}(z_i,z_{i+1}). 
\end{gather*} 
\item \label{item2:MainTheorem} There is a unique $\bar H$ such that \eqref{equation:ContinuousTimeLaxOleinkEquation} admits a solution $u$. Moreover
\begin{gather*} \label{equation:ConvolutionInteraction}
- \bar H = \lim_{\tau \to +\infty} \frac{1}{\tau} \min_{x,y \in \mathbb{R}^d} \mathcal{E}_\tau(x,y).
\end{gather*}
\item\label{item3:MainTheorem}
There exists a constant $C>0$ such that 
\[
\Big| \frac{\bar{\mathcal{L}}_\tau}{\tau} + \bar H \Big| \leq C \tau, \quad \forall \tau \in (0,1].
\]
\item\label{item4:MainTheorem}
There exist  constants $C,R>0$ such that, for every $\tau\in (0,1]$ and for every  solution  $v=u_\tau$  of \eqref{equation:DiscreteTimeLaxOleinikEquation},  or $v=u$ of \eqref{equation:ContinuousTimeLaxOleinkEquation},
\begin{enumerate}
\item $\text{Lip}(v) \leq  C$, in particular $\|v\|_\infty \leq C$ if $\min(v)=0$,
\item $\forall y\in\mathbb{R}^d$, if $ x \in \argmin_{x \in \mathbb{R}^d} \big\{ v(x) + E_\tau(x,y) \big\}$ then $ \| y-x\| \leq \tau R$.
\end{enumerate}
\item\label{item5:MainTheorem} There exist a subsequence  $\tau_i \to 0$  and a subsequence $u_{\tau_i}$ solution of \eqref{equation:DiscreteTimeLaxOleinikEquation}  such that $u_{\tau_i} \to u$ uniformly. Moreover every such  $u$ is a solution of\eqref{equation:ContinuousTimeLaxOleinkEquation}.
\end{enumerate}
\end{theorem}

Theorem \ref{theorem:Main} is proved in section \ref{section:ProofFirstTheorem}. The convergence of the discrete solution  to the solution of the ergodic cell equation has been addressed by Gomes \cite{Gomes'05} and Camilli, Cappuzzo-Dolcetta, Gomes \cite{CamCapGom'08}, but their proofs require a particular form of the Lagrangian that we do not assume. Several other numerical schemes have been studied for computing the effective Hamiltonian, see \cite{Gomes'04},  \cite{Rorro'06},  \cite{FalconeRorro'10} but the properties {\it (i)--(v)} are not stated explicitly, see also \cite{BouillardFaouZavidovique} for a mechanical Lagrangian of the  form $L(t,x,v) = W(v)+V(t,x)$. 

Note that the  discrete (backward) Lax-Oleinik equation \eqref{equation:AdditiveEigenvalueEquation} possesses a second form: the {\it discrete forward Lax-Oleinik equation},
\begin{equation*}
u_\tau(x) - \bar{E}_\tau =\max_{y\in\mathbb{R}^d} \big\{ u_\tau(y) - E_\tau(x,y) \big\}, \quad \forall x \in \mathbb{R}^d.
\end{equation*}
Theorem \ref{theorem:Main} is also valid for the forward Lax-Oleinik equation with the same effective interaction $\bar{E}_\tau$ and possibly a different solution  $u_\tau$ that is called {\it discrete forward weak KAM}. From now on we only study the backward problem.

Our second objective is to show, by introducing a discounted factor $\delta$ in the discrete Lax-Oleinik equation \eqref{equation:AdditiveEigenvalueEquation}, that we do not need to take a subsequence in time to obtain a solution of  the PDE cell equation.   A discrete version of \cite{DaviniFathiIturriagaZavidovique2014} is also proved in \cite{DFIZ2} but they  do not study the convergence issues as $\tau \to 0$. Some related results can be found in \cite{Ishii'14,Mitake'14} with a different setting.

Our approach is actually more general and applies to any short-range interaction. We first extend the definition of  the Lax-Oleinik operator.
\begin{definition}
Let $(E_\tau(x,y))_{\tau > 0}$ be a short-range interaction  satisfying (H1)--(H3). We call {\it discounted  discrete Lax-Oleinik operator}, the non-linear operator
\[
T_{\tau,\delta}[u](y) := \inf_{x\in\mathbb{R}^d} \big\{ (1-\tau\delta)u(x) + E_\tau(x,y) \big\}, \quad \forall y \in \mathbb{R}^d,
\]
defined for every $C^0$ periodic function $u$, for every $\tau > 0$ and $\delta \in (0,1]$. By coerciveness the infimum is actually attained. As before we don't change $T_{\tau,\delta}$ by using the periodic interaction $E^*_\tau(x,y)$ instead of $E_\tau(x,y)$.
\end{definition}

It is easy to show that  $T_{\tau,\delta}$ admits a unique fixed point $u_{\tau,\delta}$ that we  call {\it discounted discrete weak KAM solution}. On the other hand, it is not so easy to show  it possesses  uniform estimates  as in  proposition \ref{proposition:AprioriCompactnessLaxOleinik}, 

\begin{prop}[A priori compactness in the discounted case] \label{proposition:AprioriBoundDiscountedLaxOleinik}
Let $(E_\tau(x,y))_{\tau > 0}$ be  a  short-range interaction  satisfying  (H1)--(H6). Then there exist constants $R>1$ and $C >0$ such that, for every $\tau,\delta \in (0,1]$, 
\begin{enumerate}
\item \label{item:AprioriBoundDiscountedLaxOleinik1} $T_{\tau,\delta}$ admits a unique fixed point $u_{\tau,\delta}$ which is $C^0$ periodic,
\begin{equation*} \label{equation:DiscountedDiscreteWeakKAMsolution}
u_{\tau,\delta}(x) :=  \inf_{\substack{  (x_{-k} )_{k=0}^{+\infty} \in (\mathbb{R}^d)^{\mathbb{N}}, \, x_0=x}} \,  \sum_{k=0}^{\infty} (1-\tau\delta)^k E_{\tau}(x_{-(k+1)}, x_{-k}), \quad \forall x \in \mathbb{R}^d.
\end{equation*}
\item \label{item:AprioriBoundDiscountedLaxOleinik2} $\inf_{x,y \in \mathbb{R}^d}\dfrac{E_\tau(x,y)}{\tau\delta} \leq u_{\tau,\delta} \leq \sup_{x \in \mathbb{R}^d} \dfrac{E_\tau(x,x)}{\tau\delta}$,
\item \label{item:AprioriBoundDiscountedLaxOleinik3} $u_{\tau,\delta}$ is uniformly Lipschitz with  $\text{\rm Lip}(u_{\tau,\delta}) \leq C$, 
\item \label{item:AprioriBoundDiscountedLaxOleinik4} $\forall y\in\mathbb{R}^d$,\quad $ x\in\argmin_{x\in\mathbb{R}^d}\big\{ (1-\tau\delta)u_{\tau,\delta}(x) + E_\tau(x,y) \big\} \  \ \Rightarrow\ \|y-x\| \leq \tau R$.
\end{enumerate}
A configuration $(x_{-k})_{k=0}^\infty$ realizing the infimum in \eqref{item:AprioriBoundDiscountedLaxOleinik1} is called discounted backward calibrated configuration. Such a configuration is also calibrated for the periodic interaction $E^*_\tau(x,y)$ instead of $E_\tau(x,y)$.
\end{prop}

As in \cite{DaviniFathiIturriagaZavidovique2014} we  characterize the limit of the unique fixed point  of $T_{\tau,\delta}$ in terms of  minimizing plan, Ma\~n\'e potential. We recall these two definitions, see  \cite{Thieullen'11} for more details. 
We consider here the projection on $\mathbb{T}^d\times \mathbb{T}^d$ of objects that should be defined on $\mathbb{R}^d\times \mathbb{R}^d$ if cohomology is needed.

\begin{definition} \label{definition:PeriodicTransshipment}
A probability measure $\pi$ defined on $\mathbb{T}^d \times \mathbb{T}^d$  is said to be a {\it stationary plan} if $pr^1_*(\pi) = pr^2_*(\pi)$. (We denote by $pr^1,pr^2 : \mathbb{T}^d \times \mathbb{T}^d \to \mathbb{T}^d$, the two canonical projections).
\end{definition}

\begin{definition}\label{definition:ManePotential}
We call {\it periodic Ma\~n\'e potential}, the doubly periodic function
\begin{equation*}
\Phi^*_\tau(x,y) := \inf_{n\geq 1}  \inf_{\substack{(x_0, \ldots, x_n)\in(\mathbb{R}^d)^{n+1} \\ x_0= x, \ x_n = y}} \sum_{k=0}^{n-1} \big[ E^*_\tau(x_k,x_{k+1}) -\bar E_\tau \big], \quad \forall ~x, y\in \mathbb{R}^d.
\end{equation*}
\end{definition}

We recall how the effective Hamiltonian can be computed using stationary plan. See \cite{BernardBuffoni2007,Thieullen'11} for a proof. 
\begin{lemma} \label{lemma:ErgodicEffectiveInteraction}
Let $(E_\tau(x,y))_{\tau>0}$ be a  short-range  interaction satisfying  (H1)--(H3).  Let $E^*_\tau(x,y)$ be the associated periodic interaction. Then 
\begin{gather*}
\bar{E}_\tau = \inf \Big\{ \iint_{\mathbb{T}^d \times\mathbb{T}^d} E_\tau^*(x,y) \,\pi(dx,dy) \,:\, \pi \ \text{is a stationary plan} \ \Big\}.
\end{gather*}
\end{lemma}
Note that the infimum in lemma~\ref{lemma:ErgodicEffectiveInteraction} can be realized by compactness. We recall several classical notions.  See \cite{BernardBuffoni2007,Thieullen'11} for two distinct approaches.

\begin{definition}
Let $\pi$ be  a stationary plan on $\mathbb{T}^d \times \mathbb{T}^d$.
\begin{itemize}
\item $\pi$ is said to be {\it minimizing} if  it realizes the infimum in lemma~\ref{lemma:ErgodicEffectiveInteraction}. Define
\[
\mathcal{M}^*(E_\tau) := \{ \pi : \ \text{$\pi$ is a  minimizing plan} \}.
\]
\item $\pi$ is said to be  {\it extremal} if it is  minimizing and cannot be written as a strict  barycenter $\pi = \alpha \pi_1 +(1-\alpha)\pi_2$ of  minimizing plan,  $\pi_1$ and $\pi_2$,  with $\alpha \in (0,1)$, $\pi_1 \not= \pi_2$.
\item We call {\it Mather set}, the compact set in $\mathbb{T}^d \times \mathbb{T}^d$
\[
\text{\rm Mather}^*(E_\tau) := \overline{\cup \{\textrm{\rm supp}(\pi) : \pi \in \mathcal{M}^*(E_\tau) \}}.
\]
We call  {\it projected Mather set}, the set $pr^1(\text{\rm Mather}^*(E_\tau) )$.
\item We call {\it Aubry set}, the compact set in $\mathbb{T}^d \times \mathbb{T}^d$
\[
\text{\rm Aubry}^*(E_\tau) := \big\{ (x,y) \in \mathbb{T}^d \times \mathbb{T}^d : E^*_\tau(x,y) - \bar E_\tau +  \Phi_\tau^*(y,x)=0 \big\}.
\]
We call  {\it projected Aubry set}, the set $pr^1(\text{\rm Aubry}^*(E_\tau))$. 
\item We call {\it Aubry class}, a class of the equivalence relation  on $pr^1(\text{\rm Aubry}^*(E_\tau))$,
\[
x \sim y \Longleftrightarrow \Phi_\tau^*(x,y)+\Phi_\tau^*(y,x)=0.
\]
\end{itemize}
\end{definition}

We can show (see \cite{Thieullen'11} in the discrete setting). 

\begin{lemma} \label{lemma:ResultsManePotential}
Let $(E_\tau(x,y))_{\tau>0}$ be a  short-range  interaction satisfying  (H1)--(H3). Then
\begin{enumerate}
\item \label{item1:ResultsManePotential} $\Phi^*_\tau(x,y)$ is continuous with respect to $(x,y)$,
\item \label{item2:ResultsManePotential}  $pr^1(\text{\rm Aubry}^*(E_\tau)) = \big\{ x \in \mathbb{T}^d : \Phi_\tau^*(x,x)= 0 \big\}$,
\item \label{item3:ResultsManePotential} For any Aubry class $A$, $ \forall x,y,z \in A$,  $\Phi_\tau^*(x,y)+\Phi_\tau^*(y,z) = \Phi_\tau^*(x,z)$,
\item \label{item4:ResultsManePotential} $\text{\rm Mather}^*(E_\tau) \subset \text{\rm Aubry}^*(E_\tau)$,
\item \label{item5:ResultsManePotential} $\forall x \in pr^1(\text{\rm Aubry}^*(E_\tau))$, \  $ y \mapsto \Phi^*_\tau(x,y)$ is a discrete weak KAM solution, 
\item \label{item6:ResultsManePotential} \textbf{(representation formula)} if $u_\tau$ is any  discrete weak KAM solution, then
\[
u_\tau(y) = \min_{x \in pr^1(\text{\rm Mather}^*(E_\tau))} \{ u(x) + \Phi_\tau^*(x,y) \}, \quad \forall y \in \mathbb{R}^d.
\]
\end{enumerate}
\end{lemma}

The following lemma gives a new type of discrete weak KAM solution. Though it is simple to prove, the lemma is new and  justifies a priori the notion of balanced weak KAM solution. 

\begin{lemma} \label{lemma:AubryClassWeakKAM}
Let $(E_\tau(x,y))_{\tau >0}$ be  a  short-range interaction  satisfying  (H1)--(H3). Let $\pi$ be an extremal plan. Let $\mu = pr^1_*(\pi)$.
\begin{enumerate}
\item \label{item1:AubryClassWeakKAM} $\text{\rm supp}(\mu)$ belongs to an  Aubry class.
\item \label{item2:AubryClassWeakKAM} $y \mapsto \int \Phi_\tau^*(z,y) \, \mu(dz)$ is a discrete weak KAM solution.
\item \label{item3:AubryClassWeakKAM} $\iint \Phi_\tau^*(x,y) \, \mu(dx)\mu(dy) = 0$.
\end{enumerate}
\end{lemma}

By taking supremum or  infimum of discrete weak KAM solutions, we obtain again a discrete weak KAM solution. The balanced weak KAM solution \eqref{equation:AsymptoticallyDiscountedSolution}  is of this type. 

\begin{prop}  \label{proposition:BalancedWeakKAMsolution}
Define $\displaystyle{u_\tau^*(x) :=  \inf \Big\{ \int_{\mathbb{T}^d}\! \Phi^*_\tau(z,x) \, pr_*^1(\pi)(dz) : \pi \in \mathcal{M}^*(E_\tau) \Big\}}$. Then
\begin{enumerate}
\item \label{item1:BalancedWeakKAMsolution} $u_\tau^*$ is  a discrete weak KAM solution, 
\item \label{item2:BalancedWeakKAMsolution} 
$ u_\tau^*(y) = \sup \Big\{ w(y) : w +\bar E_\tau = T_\tau[w], \
\int_{\mathbb{T}^d } \! w(x) \, pr^1_*(\pi) (dx) \leq 0, \ \forall \pi \in \mathcal{M}^*(E_\tau) \Big\}$,

\item \label{item3:BalancedWeakKAMsolution} $\sup\{ \int u^*_\tau (y)\  pr^1_*(\pi) (dy)~:~ \pi \text{ is an extremal plan} \} =0$.
\end{enumerate}
$u_\tau^*$ is called {\it balanced discrete weak KAM solution}.
\end{prop}

The following  proposition extends to short-range interactions  the main result obtained by \cite{DaviniFathiIturriagaZavidovique2014} in the continuous case and by \cite{DFIZ2}  in the discrete case. 

\begin{prop}\label{proposition:DiscountedSelectionPrinciple}
Let $(E_\tau(x,y))_{\tau >0}$ be  a  short-range interaction  satisfying  (H1)--(H3).  Let $u_\tau^*$ be the balanced discrete weak KAM solution defined in proposition \ref{proposition:BalancedWeakKAMsolution}. Then,
\begin{gather*}
\displaystyle{\forall \tau \in (0,1], \quad \lim_{\delta \to 0} \Big(  u_{\tau,\delta} - \frac{\bar E_\tau}{\tau\delta} \Big) = u_\tau^*}, \quad \text{in the $C^0$ topology}.
\end{gather*}
\end{prop}

We summarize in the  following theorem the approximation scheme we have obtained in the case of the discrete action $\mathcal{L}_\tau(x,y)$.

\begin{theorem}[Second approximation scheme]\label{theorem:SelectionPrinciple}
Let $H(x,p)$ be a Tonelli Hamiltonian, and $L(x,v)$ be the associated Lagrangian. Let $u_{\tau,\delta}$ and $u_\delta$ be the unique $C^0$ periodic solutions of
\begin{align}
&u_{\tau,\delta}(y) = \min_{x \in \mathbb{R}^d} \big\{ (1-\tau\delta) u_{\tau,\delta}(x) + \mathcal{L}_\tau(x,y) \big\}, \quad\forall y\in\mathbb{R}^d, \ \forall \tau,\delta \in (0,1], \tag{E1} \label{equation:SelectionPrinciple1} \\
&u_\delta(y) = \inf_{\substack{\gamma \in C^{2}((-t,0],\mathbb{R}^d) \\ \gamma(0)=y}} \Big\{ e^{-t\delta}  u_\delta( \gamma(t)) + \int_{-t}^0\!\! e^{s\delta} L(\gamma(s),\dot\gamma(s)) \, ds  \Big\}, \ \forall y \in \mathbb{R}^d, \  t>0. \tag{E2}  \label{item3:SelectionPrinciple} 
\end{align}
Consider the equations with $C^0$ periodic unknowns $u_\tau$ and $u$,
\begin{align}
&u_{\tau}(y) +\bar{\mathcal{L}}_\tau = \min_{x \in \mathbb{R}^d} \big\{  u_{\tau}(x) + \mathcal{L}_\tau(x,y) \big\}, \quad\forall y\in\mathbb{R}^d, \ \forall \tau  \in (0,1], \tag{E3} \label{equation:SelectionPrinciple1} \\
&u(y)  - t \bar H = \min_{x\in\mathbb{R}^d} \big\{ u(x) + \mathcal{E}_t(x,y) \big\}, \quad \forall y \in \mathbb{R}^d, \ \forall t>0. \tag{E4} \label{equation:SelectionPrinciple4}
\end{align}
\begin{enumerate}
\item \label{item1:SelectionPrinciple} Let $\delta \in (0,1]$, $x\in\mathbb{R}^d$. Let $(x_{-n}^{\tau,\delta})_{n\geq0}$ be a backward calibrated configuration for the equation \eqref{equation:SelectionPrinciple1} starting at $x_0^{\tau,\delta}=x$. Let $\gamma_{\tau,\delta}(t)$ be the piecewise linear approximation satisfying $\gamma_{\tau,\delta}(-n\tau) = x_{-n}^{\tau,\delta}$. Then there exists a sequence $\tau_i\to0$ such that 
\begin{enumerate}
\item $\gamma_{\tau_i,\delta}(t) \to \gamma_\delta(t)$ uniformly on every compact subset of $(-\infty,0]$,
\item $\gamma_\delta \in C^{2}((-\infty,0],\mathbb{R}^d)$, \ $\|\dot\gamma_\delta \|_\infty \leq C$, \ $\| \ddot\gamma_\delta \|_\infty \leq C$
\item $u_\delta(x) = e^{-t\delta}u_\delta(\gamma_\delta(-t)) + \int_{-t}^0 e^{s\delta}L(\gamma_\delta(s),\dot\gamma_\delta(s)) \, ds, \quad \forall t\geq0$.
\end{enumerate}
\item \label{item2:SelectionPrinciple} There exists constants $C>0, R>1$ such that for every $\tau,\delta \in (0,1]$,
\begin{enumerate}
\item \label{item21:SelectionPrinciple} $u_{\tau,\delta}$ is uniformly Lipschitz with  $\text{\rm Lip}(u_{\tau,\delta}) \leq C$, 
\item \label{item22:SelectionPrinciple} $\forall y\in\mathbb{R}^d$,\quad $ x\in\argmin_{x\in\mathbb{R}^d}\big\{ (1-\tau\delta)u_{\tau,\delta}(x) + \mathcal{L}_\tau(x,y) \big\} \  \ \Rightarrow\ \|y-x\| \leq \tau R$,
\item \label{item23:SelectionPrinciple}  $\displaystyle{\| u_{\tau,\delta} - u_\delta \|_\infty \leq C \frac{\tau}{\delta} \quad\text{and}\quad  \Big\| \ \Big(u_{\tau,\delta}- \frac{\bar{\mathcal{L}_\tau}}{\tau\delta}\Big) - \Big( u_\delta+\frac{\bar H}{\delta} \Big) \ \Big\|_\infty \leq C\frac{\tau}{\delta}}$.
\end{enumerate}
\item \label{item3:SelectionPrinciple} Let $\tau \in (0,1]$ and $u_\tau^*$ be defined in proposition \ref{proposition:BalancedWeakKAMsolution}. Then
\begin{gather*}
\lim_{\delta \to 0} \Big(u_{\tau,\delta}-\frac{\bar{\mathcal{L}_\tau}}{\tau\delta} \Big) = u_\tau^*, \quad   \text{in the $C^0$ topology}.
\end{gather*}
\item \label{item4:SelectionPrinciple} Let $u^*$ be the solution of \eqref{equation:SelectionPrinciple4} defined by \eqref{equation:AsymptoticallyDiscountedSolution}. Then
\[
\lim_{\substack{\tau \to 0, \ \delta \to 0\\ \tau/\delta \to 0}} \Big( u_{\tau,\delta}- \frac{\bar{\mathcal{L}_\tau}}{\tau\delta} \Big) = u^*, \quad \text{in the $C^0$ topology}.
\]
\end{enumerate}
\end{theorem}

Theorem \ref{theorem:SelectionPrinciple}  is proved in section \ref{section:ProofSecondTheorem}. Item  \eqref{item1:SelectionPrinciple} shows how to obtain a $C^2$ minimizer in the continuous discounted case from a discrete calibrated configuration,  item \eqref{item2:SelectionPrinciple} improves similar estimates in \cite{Rorro'06,FalconeRorro'10,BouillardFaouZavidovique}. Item \eqref{item3:SelectionPrinciple} generalizes \cite{DFIZ2} and is a particular case of proposition \ref{proposition:DiscountedSelectionPrinciple},  item \eqref{item4:SelectionPrinciple} is a corollary of \eqref{item23:SelectionPrinciple} and \cite{DaviniFathiIturriagaZavidovique2014}.

\section{First approximation scheme}
\label{section:ProofFirstTheorem}

This section is devoted to the proof of theorem \ref{theorem:Main} and the necessary tools presented  before. The a priori estimates in proposition \ref{proposition:ActionComparison} are easy to prove for Tonelli Hamiltonian. We recall the following result that we admit, see \cite{Fathi'08, Mather'91} in the autonomous case, and \cite{BouillardFaouZavidovique} in the non autonomous case for more details.

\begin{lemma}[A priori compactness for minimizers] \label{lemma:AprioriBoundVelocity}
Let $H(x,p) : \mathbb{T}^d\times\mathbb{R}^d \to \mathbb{R}$ be a Tonelli Hamiltonian. For every $R>0$, there exists a constant $C(R)>0$ such that, for every $\tau>0$, $x,y \in \mathbb{R}^d$ satisfying $\|y-x\| \leq \tau R$, and for every minimizer $\gamma : [0,\tau] \to \mathbb{R}^d$ satisfying
\[
\gamma(0)=x, \ \gamma(\tau)=y, \ \int_0^\tau\! L(\gamma(s),\dot \gamma(s)) \,ds = \mathcal {E}_\tau(x,y),
\]
we have $\|\dot \gamma\| \leq C(R)$  and $\|\ddot\gamma\| \leq C(R)$.
\end{lemma}

\begin{proof}[Proof of proposition \ref{proposition:ActionComparison}]
Properties (H1)--(H6) are trivially satisfied  for the discrete action $\mathcal{L}_\tau(x,y)$. Properties (H1)--(H3) and (H5) are also easy to prove for the minimal action $\mathcal{E}_\tau(x,y)$ using the superlinearity of $L(x,v)$.

{\it Part 1: proof of property (H4).} Let $\tau>0$,  $x,y \in \mathbb{R}^d$,  $\|y-x\| \leq \tau R$. Since $\gamma(s) := x + s\frac{y-x}{\tau}$ is a particular path joining $x$ to $y$, we obtain
\[
\sup_{\tau>0, \ \|y-x\| \leq \tau R} {\textstyle\frac{1}{\tau}} \mathcal{E}_\tau(x,y) \leq \sup_{x\in\mathbb{R}^d, \ \|v\| \leq R} L(x,v).
\]
Let $\tau>0$ and $x,y\in\mathbb{R}^d$. By superlinearity, $L(x,v) \geq \|v\| -C$ for some  constant $C>0$. Then $\int_0^\tau\! L(\gamma(s),\dot\gamma(s)) \, ds \geq \|y-x\| -\tau C$ for every absolutely continuous path $\gamma:[0,\tau] \to \mathbb{R}^d$ satisfying $\gamma(0)=x$ and $\gamma(\tau)=y$. One obtains
\[
\inf_{\tau>0, \ x,y\in\mathbb{R}^d}  {\textstyle\frac{1}{\tau}} \mathcal{E}_\tau(x,y) \geq -C. 
\]
{\it Part 2: proof of property (H6).} Let $\tau \in (0,1]$,  $x,y,z \in \mathbb{R}^d$ such that $\|y-x\| \leq \tau R$ and $\|z-x\| \leq \tau R$. By Tonelli-Weierstrass, there exists a $C^2$ minimizer $\gamma : [0,\tau] \to \mathbb{R}^d$ starting at $x$, ending at $y$, and satisfying $\int_0^\tau\! L(\gamma(s),\dot\gamma(s)) \,ds = \mathcal{E}_\tau(x,y)$. Define the path $\xi : [0,\tau] \to \mathbb{R}^d$  by $\xi(s) = \gamma(s) + s \frac{z-y}{\tau}$. By lemma \ref{lemma:AprioriBoundVelocity}, there exists a constant $C(R)>0$ such that $\|\dot\gamma\| \leq C(R)$. Then
\begin{align*}
\mathcal{E}_\tau(x,z) &- \mathcal{E}_\tau(x,y) \leq \int_0^\tau\! \big[ L(\xi(s),\dot\xi(s))  -   L(\gamma(s),\dot\gamma(s)) \big] \, ds \leq \tilde C(R) \|z-y\|,
\end{align*}
where $\tilde C(R) =  \sup_{x\in\mathbb{R}^d, \ \|v\| \leq C(R)+R}\|DL(x,v)\|  $.

{\it Part 3: proof of  item \eqref{item:ComparaisonVerification}.}  Let  $R>0$ and  $C(R)$ be the  constants given by lemma \ref{lemma:AprioriBoundVelocity}. Let $\tau\in (0,1]$ and $\|y-x\| \leq \tau R$. We know that $\mathcal{E}_\tau(x,y)$ admits a $C^2$ minimizer $\gamma : [0,\tau] \to \mathbb{R}^d$ satisfying $\gamma(0)=x$, $\gamma(\tau)=y$, $\mathcal{E}_\tau(x,y) = \int_0^\tau \! L(\gamma,\dot\gamma) \,ds$, $\| \dot\gamma \| \leq C(R)$ and $\|\ddot\gamma \| \leq C(R)$. Let $V_0 = \dot\gamma(0)$. Then
\begin{gather*}
\|\gamma(s) - x\| = \| \gamma(s) - \gamma(0) \| \leq s C(R) \leq  \tau  C(R), \\ 
\| \dot\gamma(s) - V_0 \| \leq sC(R),  \ \ \Big\| \frac{y-x}{\tau} - V_0 \Big\| \leq  \tau C(R) \ \ \text{and}\ \ \Big\| \dot\gamma(s) -\frac{y-x}{\tau} \Big\| \leq 2\tau C(R).
\end{gather*}
We are now in a position to compare the two actions
\[
| \mathcal{E}_\tau(x,y) - \mathcal{L}_\tau(x,y) | \leq \int_0^\tau\! \Big| L (\gamma(s),\dot\gamma(s) ) - L \Big(x, \frac{y-x}{\tau} \Big) \Big| \,ds \leq \tau^2 \tilde C(R),
\]
with $\tilde C(R) := 2 \sup_{x \in \mathbb{R}^d, \ \|v\| \leq R+C(R)}\|DL\| \ C(R)$.
\end{proof}

The a priori estimates of proposition \ref{proposition:AprioriCompactnessLaxOleinik} are the main technical results. 

\begin{proof}[Proof of proposition \ref{proposition:AprioriCompactnessLaxOleinik}]
We begin by fixing the constants $C$ and $R$: let
\begin{align}
C_1 &:= 2 \sup_{\tau\in(0,1], \ \|y-x\|\leq \tau} \frac{E_\tau(x,y) - \bar E_\tau}{\tau}, \notag\\
R &:= \inf \Big\{ R >1 : \inf_{\tau\in(0,1], \ \|y-x\| > \tau R} \ \frac{ E_\tau(x,y) -\bar E_\tau}{\|y-x\| } > C_1 \Big\}, \label{equation:DefinitionR} \\
C &:= \max \Big( C_1, \sup_{\|y-x\|,\ \|z-x\| \leq \tau (R+1)} \frac{E_\tau(x,y) - E_\tau(x,z)}{\|z-y\|} \Big). \notag
\end{align}
Notice that $C_1$ is finite thanks to (H4), $R$ is finite thanks to (H5) and $C$ is finite thanks to (H6). 

{\it Part 1.}We  show a partial proof of item (\ref{item:AprioriLipschitzBound}), namely
\[
\|y-x\| >  \tau \ \Rightarrow \ u_\tau(y) - u_\tau(x) \leq C_1 \|y-x\|.
\]
Indeed, by choosing $n\geq 2$ such that $(n-1)\tau < \|y-x\| \leq n\tau$ and by choosing $x_i = x + \frac{i}{n}(y-x)$, we obtain $n\tau \leq 2\|y-x\|$,
\begin{gather*}
u_\tau(x_{i+1}) - u_\tau(x_i) \leq E_\tau(x_i,x_{i+1}) - \bar E_\tau, \quad\text{and} \\
u_\tau(y) - u_\tau(x) \leq n\tau \sup_{\|y-x\| \leq \tau} \frac{E_\tau(x,y)-\bar E_\tau}{\tau}  \leq C_1 \|y-x\|.
\end{gather*}

{\it Part 2.} We prove item (\ref{item:AprioriJumpBound}). Let $y \in \mathbb{R}^d$. Let $x$ be a calibrated point for $u_\tau$, that is,  $x$ satisfies
\begin{equation*}
u_\tau(y)-u_\tau(x) = E_\tau(x,y)-\bar E_\tau.
\end{equation*}
Choose some $R>1$ as in \eqref{equation:DefinitionR} and assume by contradiction that  $\|y-x\| > \tau R$. Then the first part of the proof may be used and we obtain the absurd inequality
\[
C_1 \|y-x\| \geq u_\tau(y) - u_\tau(x)  > C_1 \|y-x\|.
\]

{\it Part 3.} We end the prove of item (\ref{item:AprioriLipschitzBound}). Let $y,z \in \mathbb{R}^d$, either $\|z-y\| > \tau$ and we are done by the step 1, or $\|z-y\| \leq \tau$. Let $x$ be a calibrated point for $u_\tau$. Then $\|y-x\| \leq \tau R$, $\|z-x\| \leq \tau(R+1)$,
\begin{gather*}
u_\tau(y) - u_\tau(x) = E_\tau(x,y) - \bar E_\tau, \quad
u_\tau(z) - u_\tau(x) \leq E_\tau(x,z) - \bar E_\tau, \\
u_\tau(z) - u_\tau(y) \leq E_\tau(x,z) - E_\tau(x,y) \leq C \|z-y\|.
\end{gather*}
By permuting $z$ and $y$, we just have proved that $\text{\rm Lip}(u_\tau) \leq C$.

{\it Part 4.} We prove item (\ref{item:AprioriLaxOleinikBound}). Let $\kappa >0$. We define $R_\kappa >0$ as before
\[
R_\kappa := \inf \Big\{ R' >1 : \inf_{\tau\in(0,1], \ \|y-x\| > \tau R'} \ \frac{ E_\tau(x,y) -E_\tau(y,y)}{\|y-x\| } > \kappa \Big\}.
\] 
Let $u$ be  a periodic function satisfying $\text{Lip}(u) \leq \kappa$ and $y$ be any point in $\mathbb{R}^d$. Let  $x$ be  a point realizing the minimum of $\min_x \big\{ u(x)+E_\tau(x,y) \big\}$. Assume by contradiction that $\|y-x\| > \tau R_\kappa$, then on the one hand
\[
E_\tau(x,y) - E_\tau(y,y) > \kappa  \|y-x\|,
\]
and on the other hand $u(x)+E_\tau(x,y) \leq u(y) + E_\tau(y,y)$ and
\[
\kappa \|y-x\| \geq u(y)-u(x) \geq E_\tau(x,y) - E_\tau(y,y),
\]
which is impossible. We then estimate $\| \, T_\tau[u] - u \, \|_\infty$. On the one hand
\[
T_\tau[u](y) - u(y) \leq E_\tau(y,y).
\]
On the other hand, if $x$ realizes the minimum of $\min_{x\in\mathbb{R}^d} [u(x)+E_\tau(x,y)]$
\begin{align*}
T_\tau[u](y) - u(y) &= u(x) - u(y)  + E_\tau(x,y) \\
&\geq - \kappa \|y-x\| + \inf_{x,y\in \mathbb{R}^d} E_\tau(x,y), \\
{\textstyle\frac{1}{\tau}} \big[T_\tau[u](y) - u(y) \big] &\geq -\kappa R_\kappa + \inf_{\tau\in(0,1]} \inf_{x,y\in\mathbb{R}^d} {\textstyle\frac{1}{\tau}} E_\tau(x,y).
\end{align*}
We conclude by taking
\[
C_\kappa := \kappa R_\kappa + \sup_{\tau\in(0,1]} \sup_{y\in\mathbb{R}^d} {\textstyle\frac{1}{\tau}} E_\tau(y,y) -  \inf_{\tau\in(0,1]} \inf_{x,y\in\mathbb{R}^d} {\textstyle\frac{1}{\tau}} E_\tau(x,y). \qedhere
\]
\end{proof}

Proposition \ref{theorem:FathiWeakKamConvolutionTheorem} is new for short-range interactions. The proof we present gives another proof of the existence of Fathi's weak KAM solutions in the particular case of the minimal action.

\begin{proof}[Proof of proposition \ref{theorem:FathiWeakKamConvolutionTheorem}]

{\it Part 1.} We  prove property (\ref{item1:FathiWeakKamConvolutionTheorem}) for $\tau \in \mathbb{Q}$. Let be
\[
\bar{E}_\tau(M) := \min \Big\{ \sum_{j=1}^M E_\tau(x_{j-1},x_j) \,:\, x_j \in \mathbb{R}^d \Big\} \qquad \forall ~M\in\mathbb{Z}_+.
\]
It is enough to prove $\bar{E}_{N\tau} = N\bar{E}_\tau$ for every positive integer $N$ and $\tau>0$ not necessarily rational. We choose an integer $M>0$,
\[
(z_0,\ldots,z_M) \in \argmin \Big\{ \sum_{i=1}^M E_{N\tau}(z_{i-1},z_i) \,:\, z_i \in \mathbb{R}^d \Big\},
\]
and by min-plus convolution of $E_{N\tau}$, we choose  $(x_{i,0}, \ldots,x_{i,N})$ so that
\[
E_{N\tau}(z_{i-1},z_i) = \sum_{j=1}^N E_\tau(x_{i,j-1},x_{i,j}), \ \ x_{i,0} = z_{i-1} \ \ \text{and} \  \ x_{i,N} = z_i.
\]
Then $\bar{E}_{N\tau}(M) = \sum_{i=1}^M \sum_{j=1}^N  E_\tau(x_{i,j-1},x_{i,j}) \geq \bar{E}_{\tau}(MN)$. By dividing by $MN$ and by taking $M\to+\infty$, one obtains $\bar{E}_{N\tau} \geq N \bar{E}_\tau$. Conversely, we choose 
\[
(x_0,\ldots,x_{M-1}) \in \argmin \Big\{ \sum_{i=1}^{M-1} E_{\tau}(x_{i-1},x_i) \,:\, x_i \in \mathbb{R}^d \Big\},
\]
and $N$ integer translates $k_j \in \mathbb{Z}^d$, $j=1\ldots N$, such that $k_0=0$ and
\[
\| (x_{0}+k_j) - (x_{M-1} + k_{j-1})\| \leq 1.
\]
We define a new chain $(z_0,\ldots,z_{MN})$ by concatenating the previous translates
\[
z_{i-1+(j-1)M} := x_{i-1} + k_{j-1}M, \quad i=1,\ldots,M, \ \ j=1,\ldots,N.
\]
Then, using the fact $\|z_{jM} - z_{M-1 +(j-1)M}\| \leq 1$
\begin{gather*}
\begin{split}
N \bar{E}_\tau(M-1) &= \sum_{j=1}^{N} \sum_{i=1}^{M-1} E_\tau(z_{i-1 +(j-1)M},z_{i+(j-1)M}) \\
&\geq \sum_{j=1}^{N} \sum_{i=1}^{M} E_\tau(z_{i-1 +(j-1)M},z_{i+(j-1)M}) - N \sup_{\|y-x \| \leq 1} |E_\tau(x,y)|,
\end{split} \\
\begin{split}
\sum_{j=1}^{N} \sum_{i=1}^{M} E_\tau(z_{i-1 +(j-1)M},z_{i+(j-1)M})  &=  \sum_{i=1}^M \sum_{j=1}^N E_\tau(z_{j-1 + (i-1)N},z_{j+(i-1)N}) \\ 
&\geq \sum_{i=1}^M E_{N\tau}(z_{i-1},z_i) \geq \bar{E}_{N \tau}(M).
\end{split}
\end{gather*}
By dividing by $M$ and by taking $M\to+\infty$, one obtains $N\bar{E}_\tau \geq \bar{E}_{N\tau}$.

{\it Part 2.} We  prove an intermediate estimate, namely
\begin{gather*}
\sup_{\tau>0} \| T_\tau[0] - \bar{E}_\tau \| \leq C,
\end{gather*} 
where $C$ is the constant given by the item  (\ref{item:AprioriLipschitzBound}) of proposition \ref{proposition:AprioriCompactnessLaxOleinik}.  Let $\tau>0$ and $N$ be a positive integer such that $\tau/N \leq 1$. Let $u_{\tau/N}$ be a weak KAM solution of $T_{\tau/N}$ that we normalize by $\min u_{\tau/N}=0$. Then
\begin{gather*}
T_{\tau/N}[u_{\tau/N}] = u_{\tau/N}+ \bar{E}_{\tau/N}, \\
T_\tau[u_{\tau/N}] = (T_{\tau/N})^N[u_{\tau/N}] = u_{\tau/N}+N\bar{E}_{\tau/N} = u_{\tau/N} + \bar{E}_\tau.
\end{gather*}
Since   $\|u_{\tau/N}\| \leq C$, we obtain
\begin{gather*}
 T_\tau[0] \leq T_\tau[u_{\tau/N}] \leq C +\bar{E}\tau, \\
T_\tau[0] \geq T_\tau[u_{\tau/N}-C]  = u_{\tau/N} - C +\bar{E}_\tau  \geq -C + \bar{E}_\tau,
\end{gather*}
and finally $\|T_\tau[0] - \bar{E}_\tau\|_\infty \leq C$, for every $\tau>0$.

{\it Part 3.} We resume the proof of property (\ref{item1:FathiWeakKamConvolutionTheorem}) for $\tau\not\in \mathbb{Q}$. We choose  $p_i,q_i \in\mathbb{N}$, $q_i \to + \infty$, such that $p_i < q_i \tau < p_i+1$. Denote by $\sigma_i = p_i+1-q_i\tau$. Then $T_{p_i+1} = T_{\sigma_i} \circ T_{q_i\tau}$. Since $\|T_{q_i\tau}[0] - q_i\bar{E}_\tau\|_\infty \leq C$, by applying $T_{\sigma_i}$, one obtains on the one hand 
\[
\|T_{p_i+1}[0] -q_i\bar{E}_\tau \|_\infty \leq C +\|\,T_{\sigma_i}[0]\,\|_\infty.
\]
On the other hand $\|T_{p_i+1}[0]-(p_i+1)\bar{E}_1\|_\infty \leq C$, which implies
\[
\|(p_i+1)\bar{E}_1 - q_i\bar{E}_\tau\|_\infty \leq 2C + \sup_{\sigma\in(0,1]} \|\,T_{\sigma}[0]\,\|_\infty.
\]
Notice that item (\ref{item:AprioriLaxOleinikBound}) of proposition \ref{proposition:AprioriCompactnessLaxOleinik} implies that $\|\,T_{\sigma}[0]\,\|_\infty$ is uniformly bounded for $\sigma\in(0,1]$. We conclude by dividing by $q_i$ and letting $q_i$ go to infinity.

{\it Part 4.} We  prove item (\ref{item2:FathiWeakKamConvolutionTheorem}). From the a priori compactness property of proposition \ref{proposition:AprioriCompactnessLaxOleinik}, one can find a constant $C >0$ such that every discrete weak KAM solutions $u_{\tau}$ satisfies $\text{Lip}(u_{\tau}) \leq C$. Since $u_{\tau}$ is defined up to a constant, we may assume that $\min(u_{\tau})=0$.  By choosing a subsequence $\tau_i \to 0$, we may assume that $u_{\tau_i} \to u$ uniformly. Moreover the second part of proposition \ref{proposition:AprioriCompactnessLaxOleinik} implies that  $\|\, T_{\sigma}[v]-v\,\|_\infty \leq \sigma C$, for every $\sigma\in(0,1]$ and every Lipshitz function $v$ satisfying $\text{Lip}(v) \leq C$. Let $t>0$. There exist integers $N_i$ such that $N_i \tau_i \leq t < (N_i+1) \tau_i$. Let $\sigma_i = t-N_i\tau_i$. Then
\begin{gather*}
T_{\tau_i}[u_{\tau_i}] = u_{\tau_i} + \tau_i \bar{E}_1, \quad T_{N_i\tau_i}[u_{\tau_i}] = u_{\tau_i} +N_i\tau_i \bar{E}_1, \\
 T_{t}[u_{\tau_i}] = T_{t-N_i\tau_i}[u_{\tau_i}] +N_i\tau_i \bar{E}_1, \\ 
\| \, T_t[u_{\tau_i}] -u_{\tau_i} -t\bar{E}_1 \, \|_\infty \leq \| \, T_{\sigma_i}[u_{\tau_i}] -u_{\tau_i} \, \|_\infty + \sigma_i |\bar{E}_1|.
\end{gather*}
As $\sigma_i \to 0$, $u_{\tau_i} \to u$, $T_{\sigma_i}[u] \to u$, and $\|T_{\sigma_i}[u_{\tau_i}] - T_{\sigma_i}[u]\|_\infty \leq \|u_{\tau_i}-u\|_\infty$, we obtain $T_t[u]=u+t\bar{E}_1$.

{\it Part 5.} We prove item (\ref{item3:FathiWeakKamConvolutionTheorem}). We first notice  
\[
\min_{x,y \in \mathbb{R}^d} E_t(x,y) = \min_{y \in \mathbb{R}^d}T_t[0](y).
\]
On the one hand,
\[
T_t[0] \leq  T_t[u-\min(u)] = u+ t \bar E_1 -\min(u) \leq \max(u)-\min(u)+ t \bar E_1.
\]
On the other hand, 
\[
T_t[0] \geq T_t[u-\max(u)] = u + t \bar E_1 -\max(u) \geq \min(u)-\max(u) + t \bar E_1.
\]
In particular $\|T_t[0]- t \bar E_1\|_\infty \leq \text{\rm osc}(u)$ and $\lim_{t\to+\infty} \min_{x,y\in\mathbb{R}^d} \frac{1}{t} E_t(x,y)=\bar E_1$.
\end{proof}

We conclude this section by the proof of theorem \ref{theorem:Main}.

\begin{proof}[Proof  of theorem \ref{theorem:Main}]
{\it Part 1: proof of items \eqref{item1:MainTheorem}--\eqref{item2:MainTheorem}.}
The discrete action $\mathcal{L}_\tau(x,y)$ and the minimal action $\mathcal{E}_\tau(x,y)$ are particular cases of  short-range interactions. Item \eqref{item1:MainTheorem} is proved in proposition \ref{theorem:AdditiveEigenvalueProblem}. Item \eqref{item2:MainTheorem} is proved in proposition  \ref{theorem:FathiWeakKamConvolutionTheorem}.

{\it Part 2: Proof of item \eqref{item3:MainTheorem}.} Let us show there exists a constant $C>0$ such that
\[
| \bar{\mathcal{E}}_\tau - \bar{\mathcal{L}}_\tau | \leq \tau^2 C, \quad \forall \tau \in (0,1].
\]
Let $u_\tau$ be a discrete weak KAM solution of $\mathcal{E}_\tau(x,y)$  and $(x_{-k})_{k=0}^{+\infty}$ be a  calibrated configuration for $u_\tau$. Thanks to propositions \ref{proposition:AprioriCompactnessLaxOleinik} and \ref{proposition:ActionComparison}, there exist constants $R>0$ and $C>0$ independent of $\tau$ such that,
\begin{gather*}
\|x_{-k}-x_{-k-1}\| \leq \tau R,\quad \forall k\geq0, \\
| \mathcal{E}_\tau(x,y) - \mathcal{L}_\tau(x,y) | \leq \tau^2 C,\quad \text{$\forall x,y$ satisfying $\|y-x\| \leq \tau R$}, \\
\mathcal{E}_\tau(x_{-k-1},x_{-k}) = u_\tau(x_{-k}) - u_\tau(x_{-k-1}) + \bar{\mathcal{E}}_\tau, \\
\mathcal{L}_\tau(x_{-k-1}, x_{-k}) \leq \mathcal{E}_\tau(x_{-k-1},x_{-k})  + \tau^2 C, \\
\frac{1}{n} \sum_{k=0}^{n-1} \mathcal{L}_\tau(x_{-k-1},x_{-k}) \leq \bar{\mathcal{E}}_\tau + \tau^2 C(R) + \frac{2}{n} \|u_\tau\|_{\infty}.
\end{gather*}
By taking the limit $n\to+\infty$, and by using the mean action per site formula, we obtain $\bar{\mathcal{L}}_\tau \leq \bar{\mathcal{E}}_\tau + \tau^2 C$. By permuting the roles of $\mathcal{E}_\tau$ and $\mathcal{L}_\tau$ we conclude the proof of item \eqref{item3:MainTheorem}.

{\it Part 3: Proof of item \eqref{item4:MainTheorem}.} It  follows directly from the a priori compactness property of proposition \ref{proposition:AprioriCompactnessLaxOleinik}.

{\it Part 4: Proof of item \eqref{item5:MainTheorem}.}
We will use two Lax-Oleinik operators: $T_\tau$, the discrete Lax-Oleinik operator associated to $\mathcal{L}_\tau$, and $T^\tau$, the Lax-Oleinik semi-group associated to $\mathcal{E}_\tau$. We claim there exists  a constant $C>0$  such that, for every small $\tau > 0$, for every discrete weak KAM solution $u$ for $\mathcal{L}_\tau$, 
\[
\| \, T^\tau[u] - T_\tau[u] \, \|_\infty \leq \tau^2 C.
\]
Indeed, we know from propositions \ref{proposition:AprioriCompactnessLaxOleinik} and \ref{proposition:ActionComparison},  there exist  positive constants $R$  and $C$ such that, for every $\tau \in (0,1]$, for every discrete weak KAM solution $u$ for $\mathcal{L}_\tau$,

-- $\text{\rm Lip}(u) \leq C$, $\| u \|_\infty \leq C$,

-- $\forall y \in \mathbb{R}^d$,\quad $x \in \argmin_{x\in\mathbb{R}^d} \big\{ u(x) + \mathcal{L}_\tau(x,y) \big\} \Rightarrow \|y-x \| \leq \tau R$, 

-- $\forall y \in \mathbb{R}^d$,\quad $x \in \argmin_{x\in\mathbb{R}^d} \big\{ u(x) + \mathcal{E}_\tau(x,y) \big\} \Rightarrow \|y-x \| \leq \tau R$, 

-- $\| \, T^\tau[u]-u \, \|_\infty \leq \tau C$,

-- for every $x,y$,\quad  $\|y-x\| \leq \tau R \ \Rightarrow \| \mathcal{E}_\tau(x,y) - \mathcal{L}_\tau(x,y) | \leq \tau^2 C$.

\noindent On the one hand, for every $y$ and $x \in \argmin_{x\in\mathbb{R}^d} \big\{ u(x) + \mathcal{L}_\tau(x,y) \big\}$,
\begin{gather*}
T^\tau[u](y) \leq u(x) + \mathcal{E}_\tau(x,y) \leq u(x) + \mathcal{L}_\tau(x,y) + \tau^2 C, \\
T^\tau[u](y) \leq T_\tau[u](y) + \tau^2 C.
\end{gather*}
On the other hand, if $x \in \argmin_{x\in\mathbb{R}^d} \big[ u(x) + \mathcal{E}_\tau(x,y) \big]$,
\begin{gather*}
T^\tau[u](y) = u(x) + \mathcal{E}_\tau(x,y) \geq u(x) + \mathcal{L}_\tau(x,y) -\tau^2 C, \\
T^\tau[u](y) \geq T_\tau[u](y) - \tau^2 C.
\end{gather*}
The claim is proved. Since $\text{\rm Lip}(u)$ is uniformly bounded independently of $\tau$ for any discrete weak KAM solution $u$ for $\mathcal{L}_\tau$, we may choose  a sequence of times $\tau_i \to 0$  and  discrete weak KAM solutions $u_i$ for $\mathcal{L}_{\tau_i}$  such that $u_{i} \to u$ uniformly for some periodic Lipschitz function $u$. Let $t>0$ be fixed, and  $N_i$ be integers such that $N_i\tau_i \leq t < (N_i+1)\tau$. The non-expansiveness property of the Lax-Oleinik operator implies
\[
\| \, T^{t}[u] - T^{N_i\tau_i}[u_{i}] \, \|_\infty \leq \| \, T^{t-N_i\tau_i}[u] - u \, \|_\infty + \| u-u_i\|_\infty \to 0.
\]
The previous claim $\| \, T^{\tau_i}[u_i] - T_{\tau_i}[u_i] \, \|_\infty \leq \tau_i^2 C$ and the estimate $| \bar{\mathcal{E}}_{\tau_i} - \bar{\mathcal{L}}_{\tau_i} | \leq \tau_i^2 C$, proved in item (\ref{item3:MainTheorem}) of theorem \ref{theorem:Main},  imply 
\begin{gather*}
\| \, T^{\tau_i}[u_i] - u_i - \tau_i\bar{\mathcal{E}}_1 \, \|_\infty \leq \tau_i^2 2C.
\end{gather*}
By iterating this inequality, one obtains
\[
\| \, T^{N_i\tau_i}[u_i] - u_i - N_i\tau_i \bar{\mathcal{E}}_1 \, \|_\infty \leq N_i\tau_i^2 2C \leq t \tau_i 2C.
\]
Since $u_i + N_i\tau_i \bar{\mathcal{E}}_1 \to u + t \bar{\mathcal{E}}_1 $, one get 
\[
T^t[u] = u + t \bar{\mathcal{E}}_1, \quad \forall t>0. \qedhere
\]
\end{proof}

\begin{section}{Second approximation scheme}
\label{section:ProofSecondTheorem}

This section is devoted to the proof of  theorem \ref{theorem:SelectionPrinciple}.  Our approach follows the article \cite{DaviniFathiIturriagaZavidovique2014} to identify the selected discrete weak KAM solution but with a slightly more precise description using  Aubry classes and extremal plans.

We first  improve the a priori estimates  of proposition \ref{proposition:AprioriCompactnessLaxOleinik} to the discounted case.

\begin{proof}[Proof of proposition \ref{proposition:AprioriBoundDiscountedLaxOleinik}]
{\it Part 1.} The operator $T_{\tau,\delta}$ is contracting in $C^0$ norm, i.e.
\[
\| \, T_{\tau,\delta}[u] - T_{\tau,\delta}[v] \, \|_\infty \leq  (1-\tau\delta) \| \, u-v \, \|_\infty, \quad \forall~ u,v \in C^0(\mathbb{T}^d).
\]
Moreover, $T_{\tau,\delta}$ preserves the ball $\|u\|_\infty \leq \frac{C_0}{\delta}$ where
\[
C_0 := \sup_{\tau\in(0,1]} \Big( \sup_{x\in\mathbb{R}^d} \frac{E_\tau(x,x)}{\tau}, -\inf_{x,y\in\mathbb{R}^d} \frac{E_\tau(x,y)}{\tau} \Big).
\]
Indeed, we have
\begin{gather*}
T_{\tau,\delta}[u](y) \leq (1-\tau\delta) \max(u) +\max_{x\in\mathbb{R}^d}E_\tau(x,x), \\
T_{\tau,\delta}[u](y) \geq (1-\tau\delta)\min(u)+\min_{x,y\in\mathbb{R}^d}E_\tau(x,y), \\
\|u\|_\infty \leq \frac{C_0}{\delta} \Rightarrow \| \, T_{\tau,\delta}[u] \, \|_\infty \leq (1-\tau\delta)\|u\|_\infty +\tau C_0 \leq \frac{C_0}{\delta}.
\end{gather*}
In particular $T_{\tau,\delta}$ admits a unique fixed point $u_{\tau,\delta}$ which is inside $B(0,\frac{C_0}{\delta})$. We have proved item \eqref{item:AprioriBoundDiscountedLaxOleinik1}. The fixed point satisfies
\[
u_{\tau,\delta}(y) = \min_{x\in\mathbb{R}^d} \big\{ (1-\tau\delta)u_{\tau,\delta}(x) + E_\tau(x,y) \big\}, \quad\forall y\in\mathbb{R}^d.
\]
By iterating backward, one obtains the explicit formula for $u_{\tau,\delta}$. 

{\it Part 2.} We prove item \eqref{item:AprioriBoundDiscountedLaxOleinik3}.  We use the same reasoning as in the proof of proposition \ref{proposition:AprioriCompactnessLaxOleinik}. We claim that for every point $x,y$ satisfying $\|y-x\| \geq \tau$, we have
\begin{gather*}
|u_{\tau,\delta}(y) - u_{\tau,\delta}(x)| \leq C_1 \|y-x\|, \quad\text{with} \quad 
C_1 := \sup_{\tau\in(0,1]} \sup_{\|y-x\| \leq 2\tau} \Big( \frac{E_\tau(x,y)}{\tau}+C_0 \Big).
\end{gather*}
Indeed,  choose $n\geq 1$ so that $n\tau < \|y-x\| \leq (n+1)\tau$ and define $x_i = x+\frac{i}{n}(y-x)$. By applying  $n$ times the inequality
\[
u_{\tau,\delta}(x_{i+1}) - u_{\tau,\delta}(x_i) \leq E_\tau(x_i,x_{i+1}) + \tau\delta\|u_{\tau,\delta}\|_\infty \leq \tau C_1
\]
we obtain $u_{\tau,\delta}(y) - u_{\tau,\delta}(x) \leq C_1 \|y-x\|$. 

Define $R$ using the uniform super-linearity (H5) by
\begin{equation*}
R := \inf \Big\{ R>1 \,:\, \inf_{\tau\in(0,1]} \, \inf_{\|y-x\| \geq  \tau R} \frac{E_\tau(x,y) -C_0\tau}{\|y-x\|} > C_1
 \Big\}.\end{equation*}
We prove by contradiction that every $x\in \argmin_{x} \{ (1-\tau\delta)u_{\tau,\delta}(x) + E_\tau(x,y) \big\}$ satisfies $\|y-x\| \leq \tau R$.  If not $\|y-x\|>\tau R >\tau$,  $u_{\tau,\delta}(y) - u_{\tau,\delta}(x) \leq C_1\|y-x\|$ and by definition of $R$, we have
\[
u_{\tau,\delta}(y) - u_{\tau,\delta}(x) \geq E_\tau(x,y) - \tau\delta \|u_{\tau,\delta}\|_\infty \geq E_\tau(x,y) -\tau C_0 > C_1 \|y-x\|.
\]
We obtain a contradiction, therefore $\|y-x\| \leq \tau R$ and the proof of item \eqref{item:AprioriBoundDiscountedLaxOleinik3} is complete. 

{\it Part 3.} We prove item \eqref{item:AprioriBoundDiscountedLaxOleinik4}.  If $\|z-y\| \leq \tau$ and $x$ is a  point realizing the minimum in the definition of $u_{\tau,\delta}(y)$, 
\begin{equation*}
u_{\tau,\delta}(z) - u_{\tau,\delta}(y) \leq E_\tau(x,z) - E_\tau(x,y) \leq C\|z-y\|,
\end{equation*}
where
\[
C := \max\Big( C_1, \sup_{\tau\in(0,1]} \, \sup_{\|y-x\|,\|z-x\| \leq \tau(R+1)} \frac{E_\tau(x,z)-E_\tau(x,y)}{\|y-x\|} \Big). \qedhere
\]
\end{proof}

\begin{proof}[Proof of lemma \ref{lemma:AubryClassWeakKAM}] Let $\pi$ be an extremal plan, and $\mu=pr^1_*(\pi)$.

{\it Part 1.} Let be $\hat\Omega := (\mathbb{T}^d)^{\mathbb{N}}$, $\hat \sigma :\hat\Omega \to \hat\Omega$ be the left shift, and $pr^{1,2} : \hat\Omega \to \mathbb{T}^d \times \mathbb{T}^d$ be the projection onto the first two coordinates. We claim there exists an ergodic $\hat\sigma$-invariant probability measure $\hat\pi$ defined  on $\hat\Omega$ which projects onto $\pi$ by $pr^{1,2}$  and minimizes $\hat E_\tau(x) := E_\tau^*(x_0,x_1), \ \forall x = (x_0,x_1, \ldots) \in \hat\Omega$.

Let $\pi(dx,dy) = \mu(dx)\mathbb{P}(dy|x)$ be a regular family of disintegrated measures of $\pi$ along the projection $pr^1$. Define the Markov measure on $\hat\Omega$ by
\[
\hat{\mathbb{P}}(dx) =\mu(dx_0) \mathbb{P}(dx_1|x_0) \mathbb{P}(dx_2|x_1) \cdots
\]
Then $\hat{\mathbb{P}}$ is a $\hat\sigma$-invariant probability measure which projects onto $\pi$ and minimizes $\hat E_\tau$. Let $\hat{\mathbb{P}}(dx) = \int_{\hat \Omega}\! \hat{\mathbb{P}}_\omega(dx) \hat{\mathbb{P}}(d\omega)$ be an ergodic decomposition of $\hat{\mathbb{P}}$ (see \cite[Theorem 6.1]{Manebook}). We claim that $\omega \mapsto pr^{1,2}_*(\hat{\mathbb{P}}_\omega)$ is a.e. constant. By contradiction there would exist $\varphi \in C^0(\mathbb{T}^d \times \mathbb{T}^d)$ and a constant $a \in \mathbb{R}$ such that
\begin{equation*}
\hat B := \Big\{ \omega \in \hat\Omega :  \int \varphi(x,y) pr^{1,2}_*(\mathbb{P}_\omega)(dx,dy) < a \Big\}.
\end{equation*}
Both $\hat B$ and $\hat B^c$ have positive measure. Since $\hat{\mathbb{P}}_\omega$ is $\hat\sigma$-invariant and minimizing, $pr^{1,2}_*(\hat{\mathbb{P}}_\omega)$ is a minimizing plan. Define
\begin{align*}
\pi_1(dx,dy) := \frac{1}{\hat{\mathbb{P}}(\hat B)}\int_{\hat B} pr^{1,2}_*(\hat{\mathbb{P}}_\omega)(dx,dy) \, \hat{\mathbb{P}}(d\omega),  \\ 
\pi_2(dx,dy) := \frac{1}{\hat{\mathbb{P}}(\hat B^c)}\int_{\hat B_c} pr^{1,2}_*(\hat{\mathbb{P}}_\omega)(dx,dy) \, \hat{\mathbb{P}}(d\omega).
\end{align*}
Then $\pi_1$ and $\pi_2$ are distinct minimizing plans and 
\[
\pi = \hat{\mathbb{P}}(\hat B) \pi_1+ \hat{\mathbb{P}}(\hat B^c) \pi_2, \ \text{with} \  \hat{\mathbb{P}}(\hat B)  \in (0,1) \text{ non-trivial},
\]
which contradicts the fact that $\pi$ is extremal. We have obtained for almost every $\omega$, $pr^{1,2}(\hat{\mathbb{P}}_\omega) = \pi$  and $\hat{\mathbb{P}}_\omega$ is ergodic.

{\it Part 2: proof of item \eqref{item1:AubryClassWeakKAM}.} We have shown from part 1 there exists an ergodic  $\hat\sigma$-invariant measure $\hat\pi$ on $\hat\Omega$ such that $pr^{1}_*(\hat \pi) = \mu$, where $pr^1 : \hat\Omega \to \mathbb{T}^d$ is the first projection. Let  $\epsilon>0, x,y\in\text{supp}(\mu)$.  Define
\[
\hat{B}_x =\{  (x_0, x_1, \cdots) : x_0\in B(x, \epsilon)\},\quad \hat{B}_y =\{  (x_0, x_1, \cdots) : x_0\in B(y, \epsilon)\}.
\]
Then $\hat{B}_x, \hat{B}_y$ are open sets   and have positive measures for $\hat\pi$. Choose a discrete weak KAM solution $u_\tau$ and define
\[
\hat\varphi(z) := E^*_\tau(z_0,z_1) -[u_\tau(z_1) - u_\tau(z_0) ] -\bar{E}_\tau, \quad \forall z=(z_0,z_1, \ldots) \in \hat\Omega.
\]
By Atkinson's theorem \cite{Atkinson76}, since $\int\! \hat\varphi \, d\hat\pi = 0$,  for a.e. $z\in \hat{B}_x$, 
\[
\exists 0<m<n, \ \  \text{s.t.} \ \  \hat\sigma^m(z) \in \hat{B}_y,\ \ \hat\sigma^n(z) \in \hat{B}_x,\ \ \text{and} \ \  0 \leq \sum_{k=0}^{n-1}\hat{\varphi}\circ \hat{\sigma}^k(z) <\epsilon.
\]
We have obtained in particular,  $z_0\in B(x, \epsilon)$, $z_m\in B(y, \epsilon)$, $z_n\in B(x, \epsilon)$, and 
\[
\Phi_\tau^*(z_0, z_m) + \Phi_\tau^*(z_m, z_n) \leq \sum_{k=0}^{n-1} \hat{\varphi} \circ \hat{\sigma}^k(\omega) + [u_\tau(z_n)-u_\tau(z_0)] = O(\epsilon).
\] 
Letting $\epsilon \rightarrow 0$, we obtain $\Phi_\tau^*(x,y) + \Phi_\tau^*(y, x) = 0$ or  $x\sim y$.

{\it Part 3: proof of item \eqref{item2:AubryClassWeakKAM}.} Let $A$ be the Aubry class containing $\text{\rm supp}(\mu)$ and $\bar z \in A$ arbitrarily fixed. Then, as a function of $y$,
\[
\int\! \Phi^*_\tau(z,y) \, \mu(dz) = \int\! \Phi^*_\tau(z,\bar z) \, \mu(dz) + \Phi^*_\tau(\bar z,y), \forall y \in \mathbb{R}^d
\]
is a discrete weak KAM solution thanks to item \eqref{item5:ResultsManePotential} of lemma \ref{lemma:ResultsManePotential}.

{\it Part 4: proof of item \eqref{item3:AubryClassWeakKAM}.} For every $x,y \in A$, $\Phi^*_\tau(x,y)+\Phi_\tau^*(y,x)=0$. We conclude by integrating with respect to $\mu(dx)\mu(dy)$.
\end{proof}

\begin{proof}[Proof of proposition \ref{proposition:BalancedWeakKAMsolution}]
{\it Part 1.} We use the notations of part 1 in the proof of lemma \ref{lemma:AubryClassWeakKAM}. We claim that the infimum in the definition of $u_\tau^*$ can be realized at an extremal plan. Let $\pi$ be a minimizing plan realizing the infimum. Let $\hat{\mathbb{P}}$ be a $\hat\sigma$-invariant measure on $\hat\Omega$ such that $pr^{1,2}_*(\hat{\mathbb{P}})=\pi$. Then $\hat{\mathbb{P}}$ is minimizing. Let $\hat{\mathbb{P}}(dx) = \int\! \hat{\mathbb{P}}_\omega(dx) \, \hat{\mathbb{P}}(d\omega)$ be an ergodic decomposition. Define $\pi_\omega := pr^{1,2}_*(\hat{\mathbb{P}}_\omega)$. Since $\hat{\mathbb{P}}_\omega$ is ergodic, $\pi_\omega$ is an extremal plan. Moreover, for $x$ fixed,
\begin{align*}
\pi(dx,dy) &= \int_{\hat\Omega} \pi_\omega(dx,dy)\, \hat{\mathbb{P}}(d\omega), \\
u_\tau^*(x) &= \int_{\hat\Omega} \Big[ \int_{\mathbb{T}^d} \Phi^*_\tau(z,x) \, pr^1_*(\pi_\omega)(dz) \Big] \, \hat{\mathbb{P}}(d\omega), \\
u_\tau^*(x) &= \int_{\mathbb{T}^d} \Phi^*_\tau(z,x) \, pr^1_*(\pi_\omega)(dz), \quad  \hat{\mathbb{P}}(d\omega) \ \text{a.e.}, \\
u_\tau^*(x) &=  \inf \Big\{ \int_{\mathbb{T}^d}\! \Phi^*_\tau(z,x) \, pr_*^1(\pi)(dz) : \pi \in \mathcal{M}^*(E_\tau) \ \text{and is extremal} \Big\}.
\end{align*}

{\it Part 2: proof of items \eqref{item1:BalancedWeakKAMsolution}.} It follows from the fact that  $u_\tau^*$ is obtained as an infimum of discrete weak KAM solutions thanks to part 1 and item \eqref{item2:AubryClassWeakKAM} of lemma \ref{lemma:AubryClassWeakKAM}. 

{\it Part 3: proof of item \eqref{item2:BalancedWeakKAMsolution},\eqref{item3:BalancedWeakKAMsolution}.} They follow from item \eqref{item3:AubryClassWeakKAM} of lemma \ref{lemma:AubryClassWeakKAM}.
\end{proof}

\begin{proof}[Proof of proposition \ref{proposition:DiscountedSelectionPrinciple}]
{\it Part 1.} Let $C$ be the constant given by proposition \ref{proposition:AprioriCompactnessLaxOleinik}.  We claim that for every $\tau,\delta \in (0,1]$,
\[
\Big\| u_{\tau,\delta} - \frac{\bar E_\tau}{\tau\delta} \Big\|_\infty \leq C.
\]
Let $u_\tau$ be some discrete weak KAM solution. Let be
\[
y \in \argmax_{y\in\mathbb{R}^d} \Big\{ u_{\tau,\delta}(y) - \frac{\bar E_\tau}{\tau\delta} - u_\tau(y) \Big\}.
\]
As a fixed point of $T_{\tau,\delta}$, the discounted discrete solution satisfies for every $x$,
\begin{align*}
u_{\tau,\delta}(y) - \frac{\bar E_\tau}{\tau\delta} - u_\tau(y) &\leq (1-\tau\delta) \Big[ u_{\tau,\delta}(x) - \frac{\bar E_\tau}{\tau\delta} - u_\tau(x) \Big] \\
&\quad+ \big[ E_\tau(x,y) -u_\tau(y) + u_\tau(x) - \bar E_\tau \big] - \tau\delta u_\tau(x).
\end{align*}
Let $x$ be  a backward calibrated point for $y$ with respect to $u_\tau$ Then, 
by  definition of $y$, we have
\begin{gather*}
u_{\tau,\delta}(x) - \frac{\bar E_\tau}{\tau\delta} - u_\tau(x) \leq u_{\tau,\delta}(y) - \frac{\bar E_\tau}{\tau\delta} - u_\tau(y), \\
u_{\tau,\delta}(y) - \frac{\bar E_\tau}{\tau\delta} - u_\tau(y) \leq - u_\tau(x), \\
u_{\tau,\delta}(y) - \frac{\bar E_\tau}{\tau\delta} \leq \text{\rm osc}(u_\tau) \leq C.
\end{gather*}
On the other hand, let $y$ be a point realizing the minimum of $ u_{\tau,\delta}(y) - \frac{\bar E_\tau}{\tau\delta} - u_\tau(y)$ and  $x$ be a discounted backward calibrated  point for $y$, that is satisfying
\[
 u_{\tau,\delta}(y)  = (1-\tau\delta)u_{\tau,\delta}(x) +E_\tau(x,y).
\]
Then similar to what we have done in part 1, we obtain
\begin{align*}
u_{\tau,\delta}(y) - \frac{\bar E_\tau}{\tau\delta} - u_\tau(y) &= (1-\tau\delta) \Big[ u_{\tau,\delta}(x) - \frac{\bar E_\tau}{\tau\delta} - u_\tau(x) \Big] \\
&\quad+ \big[ E_\tau(x,y) -u_\tau(y) + u_\tau(x) - \bar E_\tau \big] - \tau\delta u_\tau(x).
\end{align*}
As $ E_\tau(x,y) -u_\tau(y) + u_\tau(x) - \bar E_\tau \geq 0$, we obtain $u_{\tau,\delta}(y) - \frac{\bar E_\tau}{\tau\delta} - u_\tau(y) \geq -u_\tau(x)$ or $u_{\tau,\delta}(y) - \frac{\bar E_\tau}{\tau\delta} \geq -\text{\rm osc}(u_\tau) \geq -C$.

{\it Part 2.}
We claim that for every $\tau,\delta \in (0,1]$,  $\pi \in \mathcal{M}^*(E_\tau)$, $\mu = pr^1_*(\pi)$,
\[
\int_{\mathbb{T}^d}\! \Big[ u_{\tau,\delta}(x) -  \frac{\bar E_\tau}{\tau\delta} \Big] \, d\mu(x) \leq 0.
\]
By definition of the discounted discrete solution $u_{\tau,\delta}$, we have
\begin{equation*}
u_{\tau,\delta} (y) \leq (1-\tau\delta)u_{\tau, \delta}(x)  + E^*_\tau(x,y), \quad \forall x,y\in\mathbb{R}^d.
\end{equation*}
By integrating the previous inequality, we obtain
\begin{multline*}
\int_{ \mathbb{T}^d}\! u_{\tau,\delta}(y) \, \mu (dy)
\leq (1-\tau\delta) \int_{\mathbb{T}^d}\! u_{\tau,\delta}(x) \, \mu(dx) + \iint_{\mathbb{T}^d \times \mathbb{T}^d}\! E^*_\tau(x,y) \, \pi(dx,dy).
\end{multline*}
The last integral is equal to $\bar E_\tau$ and $\tau\delta \int_{\mathbb{T}^d}\! u_{\tau,\delta}(x) \,\mu(dx) \leq \bar E_\tau$.

{\it Part 3.} 
Let $\tau>0$ be fixed. Let $\delta_i \to 0$ be a sequence converging to 0. For every $\delta_i$, let  $(x_{-k}^i)_{k=0}^{+\infty}$ be a discounted backward calibrated configuration,
\[
u_{\tau,\delta_i}(x_{-k}^i) = (1-\tau\delta_i) u_{\tau,\delta_i}(x_{-k-1}^i) + E_\tau(x_{-k-1}^i,x_{-k}^i).
\]
 Let $\pi_i$ be the  probability measure on $\mathbb{T}^d \times \mathbb{T}^d$ defined by
\[
\pi_i := \sum_{k\geq0}  \tau\delta (1-\tau\delta)^k \delta_{(x_{-k-1}^i, x_{-k}^i )}.
\]
We claim that every weak${}^*$ accumulation measure $\pi$ of $\{\pi_i\}_{i=1}^{\infty}$ is a minimizing plan. Assume that $\pi_i \to \pi$ as $i \to \infty$ to simplify the notations. 

We first prove that $\pi$ is a stationary plan. Let $\varphi : \mathbb{T}^d \to \mathbb{R}$ be a continuous function, then
\begin{align*}
\iint_{\mathbb{T}^d\times\mathbb{T}^d} \varphi(y) \, \pi_i(dx,dy) &= \sum_{k\geq0} \tau\delta_i(1-\tau\delta_i)^k \varphi(x_{-k}^i) \\
&=\tau\delta_i \varphi(x_0^i) +  (1-\tau\delta_i)\sum_{k\geq0} \tau\delta_i(1-\tau\delta_i)^k \varphi(x_{-k-1}^i) \\
&= \tau\delta_i \varphi(x_0^i) + (1-\tau\delta_i) \iint_{\mathbb{T}^d\times\mathbb{T}^d} \varphi(x) \, \pi_i(dx,dy).
\end{align*}
We complete the proof by letting $\delta_i \to 0$. We next prove that $\pi$ is minimizing:
\begin{align*}
\iint_{\mathbb{T}^d\times\mathbb{T}^d}\! E^*_\tau(x, &y) \, \pi_i(dx,dy) = \sum_{k\geq 0}^{} \tau\delta_i(1-\tau\delta_i)^k E^*_\tau(x_{-k-1}^i,x_{-k}^i) \\
&= \sum_{k\geq 0}^{} \tau\delta_i(1-\tau\delta_i)^k \big[ u_{\tau,\delta_i}(x_{-k}^i) - (1-\tau\delta_i) u_{\tau,\delta_i}(x_{-k-1}^i) \big] = \tau\delta_i u_{\tau,\delta_i} (x_0^i).
\end{align*}
We conclude the proof thanks to part 1 which implies $\tau\delta_iu_{\tau\delta_i} \to \bar E_\tau$ uniformly.

{\it Part 4.} Since $\text{\rm Lip}(u_{\tau,\delta})$ and $\big\| u_{\tau,\delta} - \frac{\bar E_\tau}{\tau\delta} \big\|_\infty$ are uniformly bounded with respect to $\delta$, there exists a sub-sequence $\delta_i \to 0$ and a $C^0$ periodic function $u_\tau$ such that, 
\[
u_{\tau,\delta_i} - \frac{\bar E_\tau}{\tau\delta_i} \to u_\tau, \quad \text{ in the $C^0$-topology}.
\]

We first prove  that $u_\tau$ is a discrete weak KAM solution. On the one hand, by letting $\delta_i \to 0$ in
\[
u_{\tau,\delta_i}(y) - \frac{\bar E_\tau}{\tau\delta_i} \leq (1-\tau\delta_i) \big[ u_{\tau,\delta_i}(x) - \frac{\bar E_\tau}{\tau\delta_i} \big] + E_\tau(x,y) - \bar E_\tau,
\]
one obtains $u_\tau(y) - u_\tau(x) \leq E_\tau(x,y) - \bar E_\tau$, for every $x,y \in \mathbb{R}^d$. On the other hand, for every $y$, there exists $x_i \in \mathbb{R}^d$ such that
\[
u_{\tau,\delta_i}(y) - \frac{\bar E_\tau}{\tau\delta_i} = (1-\tau\delta_i) \Big[ u_{\tau,\delta_i}(x_i) - \frac{\bar E_\tau}{\tau\delta_i} \Big] + E_\tau(x_i,y) - \bar E_\tau.
\]
Proposition \ref{proposition:AprioriBoundDiscountedLaxOleinik} implies there exists  a constant $R>0$, independent of $\delta$, such that  $\|y-x_i\| \leq \tau R$. By taking possibly a sub-sequence, one may assume $x_i \to x$ for some $x\in\mathbb{R}^d$. One then obtains $u_\tau(y) - u_\tau(x) = E_\tau(x,y) - \bar E_\tau$. The proof is finished.

We next prove that $u_\tau = u_\tau^*$ given by proposition \ref{proposition:BalancedWeakKAMsolution}.  Let $\pi \in\mathcal{M}^*(E_\tau)$ and $\mu = pr^1_*(\pi)$.  By letting $\delta_i \to 0$ in part 2, one obtains  $\int_{\mathbb{T}^d}\! u_\tau(x)\, d\mu(x) \leq  0$ and
\[
u_\tau(y) \leq \sup \Big\{ w(y) \,:\, T_\tau[w] = w + \bar E_\tau,  \ \int_{\mathbb{T}^d}\! w(x) \, pr^1_*(\pi)(dx) \leq 0, \ \forall \pi \in \mathcal{M}^*(E_\tau) \Big\}.
\]
Conversely, let $w$ be a discrete  weak KAM solution satisfying $\int_{\mathbb{T}^d}\! w \, dpr^1_*(\pi) \leq 0$ for every $\pi \in \mathcal{M}^*(E_\tau)$. Let $y\in\mathbb{R}^d$ and for every $\delta_i$,  $(x_{-k}^i)_{k\geq0}$ be a discounted backward calibrated configuration starting at $y=x_0^i$. Then
\begin{multline*}
u_{\tau,\delta_i}(x_{-k}^i) -\frac{\bar E_\tau}{\tau\delta_i}  - w(x_{-k}^i) = (1-\tau\delta_i) \Big[ u_{\tau,\delta_i}(x_{-k-1}^i)  - w(x_{-k-1}^i) -\frac{\bar E_\tau}{\tau\delta_i} \Big] \\
+\big[ E_\tau(x_{-k-1}^i, x_{-k}^i) - w(x_{-k}^i)  + w(x_{-k-1}^i) -\bar E_\tau  \big]  -\tau\delta_i w(x_{-k-1}^i).
\end{multline*}
As  $E_\tau(x_{-k-1}^i, x_{-k}^i) - w(x_{-k}^i)  + w(x_{-k-1}^i) -\bar E_\tau  \geq 0$, by iterating these inequalities, one obtains
\[
u_{\tau,\delta_i}(y) -\frac{\bar E_\tau}{\tau\delta_i}   - w(y)  \geq \sum_{k\geq0} -\tau\delta_i (1-\tau\delta_i)^k w(x_{-k-1}^i) = -\iint_{\mathbb{T}^d\times\mathbb{T}^d}\! w(x) \, \pi_i(dx,dy),
\]
where $\pi_i$ is the probability measure  defined in part 3. As $\pi_i$ converges to a minimizing plan $\pi$, one obtains $u_\tau(y) - w(y) \geq - \int_{\mathbb{T}^d}\! w \, dpr^1_*(\pi) \geq 0$ and therefore $u_\tau \geq u_\tau^*$. Since $u_\tau^*$ is the  only accumulation point of $u_{\tau,\delta} -\frac{\bar E_\tau}{\tau\delta}$, the proof of proposition  \ref{proposition:DiscountedSelectionPrinciple} is complete.
\end{proof}

The only results in theorem \ref{theorem:SelectionPrinciple} to be proved are items \eqref{item1:SelectionPrinciple} and \eqref{item23:SelectionPrinciple}. Items \eqref{item21:SelectionPrinciple}--\eqref{item22:SelectionPrinciple}  are particular cases of proposition \ref{proposition:AprioriBoundDiscountedLaxOleinik}. Item \eqref{item3:SelectionPrinciple} is a particular case of proposition \ref{proposition:DiscountedSelectionPrinciple}. Item \eqref{item4:SelectionPrinciple} is a consequence of item \eqref{item23:SelectionPrinciple} and the existence of the balanced weak KAM solution \eqref{equation:AsymptoticallyDiscountedSolution}.

\begin{proof}[Proof of item \eqref{item1:SelectionPrinciple} of theorem \ref{theorem:SelectionPrinciple}]
{\it Part 1.} Let  $\tau>0$, and $\{x_{n}^{\tau,\delta}\}_{n\leq0}$ be a discounted backward calibrated configuration for the discrete action  $\mathcal{L}_\tau$  ending at $x$. We note 
\[
v_{n}^{\tau,\delta} := \frac{1}{\tau} \Big( x_{n+1}^{\tau,\delta} - x_n^{\tau,\delta} \Big), \quad \forall n \leq -1.
\]
We show in this part there exists a constant $C>0$, independent of $n,\delta$ and $x$, such that $\|v_{n}^{\tau,\delta} - v_{n-1}^{\tau,\delta} \| \leq C\tau$ for all $n\leq -1$. Let $x_n := x_{n}^{\tau,\delta}$ and $v_n := v_n^{\tau,\delta}$. By the definition of calibration  we have
\begin{align*}
u_{\tau,\delta}(x_{n+1}) &= (1-\tau\delta) u_{\tau,\delta}(x_n) + \mathcal{L}_\tau(x_{n},x_{n+1}) \\
&= (1-\tau\delta)^2 u_{\tau,\delta}(x_{n-1}) + (1-\tau\delta) \mathcal{L}_\tau(x_{n-1},x_n) + \mathcal{L}_\tau(x_n,x_{n+1}) \\
&\leq (1-\tau\delta) u_{\tau,\delta}(x) +\mathcal{L}_\tau(x,x_{n+1}), \quad \forall x \in \mathbb{R}^d \\
&\leq (1-\tau\delta)^2 u_{\tau,\delta}(x_{n-1}) + (1-\tau\delta) \mathcal{L}_\tau(x_{n-1},x) + \mathcal{L}_\tau(x,x_{n+1}), \quad \forall x \in \mathbb{R}^d.
\end{align*}
In other words $\{x_n^{\tau,\delta}\}_{n\leq0}$ is minimizing in the following sense
\[
(1-\tau\delta) \mathcal{L}_\tau(x_{n-1},x_n) + \mathcal{L}_\tau(x_n,x_{n+1}) \leq (1-\tau\delta) \mathcal{L}_\tau(x_{n-1},x) + \mathcal{L}_\tau(x,x_{n+1}), \quad \forall x \in \mathbb{R}^d,
\]
and satisfies the {\it discounted discrete Euler-Lagrange equation}
\begin{align}
&\quad (1-\tau\delta) \frac{\partial \mathcal{L}_\tau}{\partial y}(x_{n-1},x_n) + \frac{\partial \mathcal{L}_\tau}{\partial x}(x_n,x_{n+1}) = 0 \notag \\
\Longleftrightarrow&\quad (1-\tau\delta) \frac{\partial L}{\partial v}(x_{n-1}, v_{n-1}) - \frac{\partial L}{\partial v}(x_n, v_n) + \tau \frac{\partial L}{\partial x}(x_n, v_n) = 0 \notag \\
\Longleftrightarrow&\quad  \frac{1}{\tau} \Big[ \frac{\partial L}{\partial v}(x_n, v_n)  - \frac{\partial L}{\partial v}(x_{n-1}, v_{n-1}) \Big] = \frac{\partial L}{\partial x}(x_n, v_n) - \delta  \frac{\partial L}{\partial v}(x_{n-1}, v_{n-1}).
\end{align}
Proposition \ref{proposition:AprioriBoundDiscountedLaxOleinik} shows there exists  $R>0$ such that $\|v_n^{\tau,\delta}\| \leq R$, $\forall n\leq-1$. The property of positive definiteness (L1) implies the existence of a constant $\alpha(R)>0$ such that, for every $ x\in\mathbb{R}^d$, $v\in\mathbb{R}^d$  satisfying $\|v\| \leq R$,
\[
\frac{\partial^2 L}{\partial v \partial v}(x,v).(h,h) \geq \alpha(R) \|h\|^2, \quad \forall  h\in \mathbb{R}^d.
\]
By integrating over $t \in [0,1]$ the term
$
\frac{d}{dt} \Big( \frac{\partial L} {\partial v} \big( x_{n-1}+t(x_n-x_{n-1}), v_{n-1}+t(v_n-v_{n-1})\big) \Big)
$
and by taking the scalar product with $(v_n-v_{n-1})$, one obtains
\[
\alpha(R) \|v_n-v_{n-1}\| \leq \Big\|\frac{\partial L}{\partial x\partial v}\Big\| \ \| x_n-x_{n-1}\| + \tau \Big( \ \Big\|\frac{\partial L}{\partial x}\Big\| + \delta \Big\|\frac{\partial L}{\partial v} \Big\| \Big)
\]
where all norms $\|\cdot\|$ are taken over $\mathbb{T}^d \times \big\{ v \in \mathbb{R}^d \,:\,  \ \|v\| \leq R\| \big\}$. As  $\| x_{n}- x_{n-1}\| \leq \tau R$ thanks to item \eqref{item22:SelectionPrinciple} of proposition \ref{proposition:AprioriBoundDiscountedLaxOleinik}, one obtains $\| v_n - v_{n-1} \| \leq \tau C$, for some constant $C>0$, uniformly in $n,\delta$  and $x$. 

\medskip
{\it Part 2.} Let $\gamma_{\tau,\delta}^x : (-\infty,0] \to \mathbb{R}^d$ be the piecewise affine path interpolating the points $x_{n}$ at time $n\tau$. We show that $\gamma_{\tau,\delta}^x$ is  Lipschitz uniformly in $n,\delta$ and $\{x_n^{\tau,\delta}\}_{n\leq0}$. To simplify we write $\gamma= \gamma_{\tau,\delta}^x$. Let $s<t<0$. Either $s,t$ belong to the same interval $((n-1)\tau,n\tau]$. As $\gamma$ is affine with speed bounded by $R$, we obtain $\| \gamma(t) - \gamma(s) \| \leq |t-s| R$. Or $s,t$ belong to different  intervals. By introducing the points $x_{n}$ corresponding to the intermediate  times $s \leq n\tau \leq t$, one obtains again the same estimate.

{\it Part 3.} We choose a subsequence $\tau_i \to 0$ and a discounted backward calibrated configuration  $\{x_n^i\}_{n \leq 0}$ such that $\gamma_i := \gamma_{\tau_i,\delta}^x \to \gamma_{\delta}^x$ uniformly on any compact interval of $(-\infty,0]$ for some Lipschitz function $\gamma_{\delta}^x$. We claim there exists a uniformly Lipschitz function $V : (-\infty,0] \to \mathbb{R}^d$ such that 
\[
\int_{t}^0\! V(s) \, ds = x-\gamma_{\delta}^x(t), \quad \forall t \leq 0.
\]
Let $T\subset(-\infty,0)$ be a countable dense subset. Let be  $V_i : (-\infty,0) \to \mathbb{R}^d$ such that
\[
V_i(t)  :=\frac{1}{\tau_i} \Big(x_{n }^{i} - x_{n-1}^{i} \Big), \quad \forall t \in [(n-1)\tau_i, n\tau_i) ,\ \forall n\leq0.
\]
By compactness of the ball $\{ v \,:\, \|v\| \leq R\}$, by taking a subsequence if needed, we may assume $V_i(t)\to V(t)$ exists for every $t \in T$. Let $s<t<0$ and $m \leq n$  be non positive integers such that $(m-1)\tau_i \leq  s < m\tau_i$ and $(n-1)\tau_i \leq  t < n\tau_i$. Part 1 implies, 
\begin{gather*}
\| V_i(t) - V_i(s) \| = \| v_{n-1}^i - v_{m-1}^i \| \leq (n - m)\tau_i C \leq |t-s| C + \tau_i C.
\end{gather*}
By letting $\tau_i \to 0$, one obtains  $\|V(t)-V(s) \| \leq |t-s| C$ for every $s,t \in T$. Let  $V: (-\infty,0) \to \mathbb{R}^d$ be the unique Lipschitz extension of $V$. Then $V_i(t) \to V(t)$ for every $t \in (-\infty,0)$.  Since
\[
\int_t^0\! V_i(s) \, ds = x- \gamma_{i}(t), \quad \forall t<0,
\] 
the claim is proved and $\gamma_{\delta}^x$ is a  $C^{1,1}$ path.

\medskip
{\it Part 4.} Item \eqref{item21:SelectionPrinciple} of proposition \ref{proposition:AprioriBoundDiscountedLaxOleinik} shows there exists a constant $C>0$ such that $\textrm{\rm Lip}(u_{ \tau_i,\delta}) \leq C$. By taking a subsequence if necessary, we may assume that $u_i := u_{\tau_i,\delta} \to u$ uniformly for some Lipschitz function $u$. We claim that
\begin{gather*}
u(x) - e^{t\delta}u(\gamma_{\delta}^x(t)) = \int_t^0\! e^{s\delta} L(\gamma_{\delta}^x(s) , \dot\gamma_{\delta}^x(s)) \, ds,\quad   \forall x \in \mathbb{R}^d, \  \forall t\leq0.
\end{gather*}
Indeed using the notations in part 3, we have for every $n\leq-1$,
\[
u_i(x) = (1-\tau_i\delta)^{-n} u_i \circ \gamma_i(n\tau_i) + \sum_{k=n}^{-1} (1-\tau_i\delta)^{-k-1} \tau_i  L \big(\gamma_i(k\tau_i),V_i( k\tau_i) \big).
\]
Let $t<0$ be fixed,  $n\leq 0$ be such that $(n-1)\tau_i \leq t < n\tau_i$. Then
\begin{align*}
I := \Big| \sum_{k=n}^{-1} (1-\tau_i\delta)^{-k-1} \tau_i  L \big(\gamma_i(k\tau_i),V_i( k\tau_i) \big) - \int_{n\tau_i}^0 e^{s\delta}  L(\gamma_i(s) ,V_i(s)) \,ds\ \Big|
\end{align*}
can be bounded from above by the following three terms $I_1,I_2,I_3$
\begin{align*}
I_1 &= \sum_{k=n}^{-1 } (1-\tau_i\delta)^{-k-1}  \int_{k\tau_i}^{(k+1)\tau_i} \big| L \big(\gamma_i(k\tau_i),V_i( k\tau_i) \big) - L(\gamma_i(s), V_i(s) ) \big| \,ds  \\
&\leq R\Big\| \frac{\partial L}{\partial x} \Big\| \frac{\tau_i}{\delta}, \\
I_2 &= \sum_{k=n}^{-1} \Big[  (1-\tau_i\delta)^{-k-1} -(1-\tau_i\delta)^{-k} \Big]  \int_{k\tau_i}^{(k+1)\tau_i} \big|  L(\gamma_i(s), V_i(s) ) \big| \,ds  \\
&\leq \tau_i \| L \| \  \Big( 1- (1-\tau_i\delta)^{-n} \Big) \leq \tau_i \| L \|, \\
I_3 &= \sum_{k=n}^{-1} \int_{k\tau_i}^{(k+1)\tau_i} \Big[ e^{s\delta} -(1-\tau_i\delta)^{-k} \Big] \ \big|  L(\gamma_i(s), V_i(s) ) \big| \,ds  \\
&\leq \|L \| \  \Big[ \int_{n\tau_i}^0 e^{s\delta} \, ds - \tau_i \sum_{k=n}^{-1} (1-\tau_i\delta)^{-k} \Big] \leq  \tau_i \| L \|.
\end{align*}
We finally obtain 
\[
I \leq R\Big\| \frac{\partial L}{\partial x} \Big\| \frac{\tau_i}{\delta} +  2\tau_i \| L \|,
\]
and   the  claim is proved by letting $\tau_i \to 0$, since $n\tau_i \to t$, $u_i \to u$ uniformly on $\mathbb{R}^d$,  and both $\gamma_i \to \gamma_\delta^x$ and $V_i \to \dot\gamma_\delta^x$  uniformly on any compact set of $(-\infty,0]$. 

\medskip
{\it Part 5.} We claim that
\begin{gather*}
u(x) -e^{-t\delta} u(x-tv) \leq \int_{-t}^0\! e^{s\delta} L(x+sv,v) \,ds, \quad    \forall x \in \mathbb{R}^d, \  \forall t\geq0, \ \forall v \in \mathbb{R}^d.
\end{gather*}
We choose as before $n\leq0$ such that $(n-1)\tau_i \leq t < n\tau_i$. Let  $x_k^i := x-k\tau_i v$, $\forall k \in\{n,\ldots,-1,0\}$. By definition of  $u_i = u_{\tau_i,\delta}$, we have
\[
u_i(x) \leq (1-\tau_i\delta)^{-n} u_i(x_n^i) + \sum_{k=n}^{-1} (1-\tau_i\delta)^{-k-1} \tau_i  L (x_k^i,v).
\]
Then the expression $| \sum_{k=n}^{-1} (1-\tau_i\delta)^{-k-1} \tau_i  L (x_k^i,v) - \int_{n\tau_i}^0 e^{s\delta}  L(x+sv ,v) \,ds |$
is estimated in the same way as before, and  the claim is proved.

\medskip
{\it Part 6.} By approximating any $C^2$ path picewise linearly, we obtain that, for any $\gamma \in C^2((-\infty,0],\mathbb{R}^d)$ ending at $\gamma(0)=x$,
\begin{gather*}
u(x) -e^{-t\delta} u(\gamma(-t)) \leq \int_{-t}^0\! e^{s\delta} L(\gamma(s),\dot\gamma(s)) \,ds, \quad    \forall x \in \mathbb{R}^d, \  \forall t\geq0, \ \forall v \in \mathbb{R}^d.
\end{gather*}
We just have proved that $u$ is unique given by \eqref{equation:DiscoutedProgrammingPrinciple}, and  that $\gamma_\delta^x$ is a  $C^2$ minimizer  by Tonelli Weierstrass theorem.
\end{proof}

\begin{proof}[Proof of item \eqref{item23:SelectionPrinciple} of theorem \ref{theorem:SelectionPrinciple}]
We first show $u_{\tau,\delta}  -u_\delta \leq C\frac{\tau}{\delta}$.  Thanks to item \eqref{item1:SelectionPrinciple}, there exists a constant $C_1>0$ such that, for every $x \in \mathbb{R}^d$,  there exists a $C^{1,1}$ curve $\gamma_\delta^x : (-\infty,0] \to \mathbb{R}^d$,  satisfying $\gamma_\delta^x(0)=x$, $\| \dot \gamma_\delta^x \| \leq C_1$ and  $\text{\rm Lip}(\dot \gamma_\delta^x) \leq C_1$ uniformly on $(-\infty,0]$, and 
\[
u_\delta(x) = \int_{-\infty}^0\! e^{s\delta} L(\gamma_\delta^x(s),\dot\gamma_\delta^x(s)) \,ds.
\]
Let $x_{-k}:= \gamma_\delta^x(-k\tau)$, $v_{-k} := (x_{-k+1}-x_{-k})/ \tau$, for every $k\geq0$. Then
\begin{gather*}
u_{\tau,\delta}(x) \leq \sum_{k\geq0} (1-\tau \delta)^{k} \mathcal{L}_\tau(x_{-k-1},x_{-k}), \\
\begin{split}
(1-\tau\delta) u_{\tau,\delta}(x) - u_\delta(x) &\leq \sum_{k\geq 0} \int_{-(k+1)\tau}^{-k\tau} \Big[ (1-\tau \delta)^{k+1}-e^{s\delta} \Big] L(x_{-k-1},v_{-k-1}) \\
&+ \sum_{k\geq0} \int_{-(k+1)\tau}^{-k\tau} e^{s\delta} \big[ L(x_{-k-1}, v_{-k-1}) - L(\gamma_\delta(s), \dot\gamma_\delta(s) ) \big] \,ds.
\end{split}
\end{gather*}
For every $s \in [-(k+1)\tau,-k\tau]$, 
\begin{gather*}
\| \gamma_\delta(s) - x_{-k-1} \| \leq C_1\tau, \quad \| \dot\gamma_\delta(s) - v_{-k-1} \| \leq C_1\tau, \\
| L(x_{-k-1}, v_{-k-1}) - L(\gamma_\delta(s), \dot\gamma_\delta(s) ) | \leq \| DL \|_\infty C_1 \tau,
\end{gather*}
(where $\| DL \|_\infty$ is computed by taking the supremum of $\|DL(x,v)\|_\infty$ over $x\in\mathbb{R}^d$ and $\|v\| \leq C_1$). Moreover 
\begin{gather*}
	\sum_{k\geq 0} \int_{-(k+1)\tau}^{-k\tau} \Big[ e^{s\delta} - (1-\tau \delta)^{k+1} \Big] \leq \frac{1}{\delta} - \frac{\tau(1-\tau\delta)}{\tau\delta} = \tau.
\end{gather*}
Let $\|L\|_\infty$ be the supremum of $L(x,v)$ over $x\in\mathbb{R}^d$ and $\|v\| \leq C_1$. Then item \eqref{item:AprioriBoundDiscountedLaxOleinik2} of proposition \ref{proposition:AprioriBoundDiscountedLaxOleinik}  implies 
\[
u_{\tau,\delta}(x) - u_\delta(x) \leq 2\| L \|_\infty\tau + \|DL \|_\infty C_1\frac{\tau}{\delta} \leq   \big( 2 \| L \|_\infty + \|DL \|_\infty C_1 \big) \frac{\tau}{\delta} := C \frac{\tau}{\delta}.
\]

We next show $u_{\tau,\delta}- u_\delta \geq -C \frac{\tau}{\delta}$. Let $x\in\mathbb{R}^d$ and $\{x_{-k}\}_{k\geq0}$  a discounted backward calibrated configuration for $\mathcal{L}_\tau$ starting at $x$, then
\[
u_{\tau,\delta}(x) = \sum_{k\geq0} (1-\tau\delta)^k \mathcal{L}_\tau(x_{-k-1},x_{-k}).
\]
Let $\gamma:(-\infty,0] \to \mathbb{R}^d$ be the piecewise linear path interpolating the points $x_{-k}$ at the times $-k\tau$. Then,  property \eqref{equation:DiscoutedProgrammingPrinciple} implies
\[
u_\delta(x) \leq \int_{-\infty}^0\! e^{s\delta} L(\gamma(s),\dot\gamma(s)) \,ds.
\] 
Using item \eqref{item22:SelectionPrinciple} of proposition \ref{proposition:AprioriBoundDiscountedLaxOleinik}, we notice that for every $s \in [-(k+1)\tau,-k\tau]$, 
\begin{gather*}
\| \gamma(s) - x_{-k-1} \| \leq \| x_{-k} - x_{-k-1} \| \leq R\tau, \quad \dot\gamma(s) = (x_{-k}-x_{-k-1})/\tau := v_{-k-1}, \\
| L(x_{-k-1}, v_{-k-1}) - L(\gamma(s), \dot\gamma(s) ) | \leq \Big\|\frac{\partial L}{\partial x} \Big\|_\infty R \tau,
\end{gather*}
(where $\big\|\frac{\partial L}{\partial x} \big\|_\infty $ is computed by taking the supremum of $\big\|\frac{\partial L}{\partial x}(x,v)\big\|$ over $x\in\mathbb{R}^d$ and $\|v\| \leq R$). Let $C_3 := \inf_{x,v \in \mathbb{R}^d} L(x,v)$. Then item \eqref{item:AprioriBoundDiscountedLaxOleinik2}  of proposition \ref{proposition:AprioriBoundDiscountedLaxOleinik} implies
\[
u_{\tau,\delta}(x) - u_\delta(x) \geq  \Big(    C_3 - \| L \| - \Big\|\frac{\partial L}{\partial x}\Big\|_\infty R \Big) \frac{\tau}{\delta} := - C \frac{\tau}{\delta}. \qedhere
\]
\end{proof}

\end{section}

\section*{Acknowledgement}
We would like to thank A. Fathi for the opportunity to have an access to his preprint \cite{DaviniFathiIturriagaZavidovique2014}. At the time our paper was written, we did not know the existence of \cite{DFIZ2}. X. Su is supported by both National
Natural Science Foundation of China (Grant No. 11301513) and ``the Fundamental Research Funds for the Central Universities". Ph. Thieullen is supported by  ANR WKBHJ ANR-12-BS01-0020.

\bibliographystyle{alpha}
\bibliography{discrete-weak-KAM}

\begin{thebibliography}{AAOIM14}

\bibitem[AAOIM14]{Ishii'14}
Eman~S. Al-Aidarous, Alzahrani~Ebraheem O., Hitoshi Ishii, and Younas Arshad~M.
  M.
\newblock A convergence result for the ergodic problem for hamilton-jacobi
  equations with neumann type boundary conditionss.
\newblock 2014.

\bibitem[ALD83]{Aubry'83}
Serge Aubry and P.Y. Le~Daeron.
\newblock The discrete {F}renkel-{K}ontorova model and its extensions: I. exact
  results for the ground states.
\newblock {\em Physica D}, 8:381--422, 1983.

\bibitem[Atk76]{Atkinson76}
Giles Atkinson.
\newblock Recurrence of co-cycles and random walks.
\newblock {\em J. London Math. Soc. (2)}, 13(3):486--488, 1976.

\bibitem[Bar94]{Barles'94}
Guy Barles.
\newblock {\em Solutions de viscosit\'e des \'equations de
  {H}amilton-{J}acobi}.
\newblock Springer, 1994.

\bibitem[Bar13]{Barles2013}
Guy Barles.
\newblock An {I}ntroduction to the {T}heory of {V}iscosity {S}olutions for
  {F}irst-{O}rder {H}amilton-{J}acobi {E}quations and {A}pplications.
\newblock 2074:49--109, 2013.

\bibitem[BB07]{BernardBuffoni2007}
Patrick Bernard and Boris Buffoni.
\newblock Weak {KAM} pairs and {M}onge-{K}antorovich duality.
\newblock In {\em Asymptotic analysis and singularities---elliptic and
  parabolic {PDE}s and related problems}, volume~47 of {\em Adv. Stud. Pure
  Math.}, pages 397--420. Math. Soc. Japan, Tokyo, 2007.

\bibitem[BCD97]{BardiDolcetta'97}
Martino Bardi and Italo Capuzzo-Dolcetta.
\newblock {\em Optimal control and viscosity solutions of
  {H}amilton-{J}acobi-{B}ellman equations}.
\newblock Systems \& Control: Foundations \& Applications. Birkh\"auser Boston,
  Inc., Boston, MA, 1997.
\newblock With appendices by Maurizio Falcone and Pierpaolo Soravia.

\bibitem[BFZ16]{BouillardFaouZavidovique}
Anne Bouillard, Erwan Faou, and Maxime Zavidovique.
\newblock Fast weak-{KAM} integrators for separable {H}amiltonian systems.
\newblock {\em Math. Comp.}, 85(297):85--117, 2016.

\bibitem[CCDG08]{CamCapGom'08}
Fabio Camilli, Italo Cappuzzo~Dolcetta, and Diogo~A. Gomes.
\newblock Error estimates for the approximations of the effective
  {H}amiltonian.
\newblock {\em Appl. Math. Optim.}, 57:30--57, 2008.

\bibitem[CG86]{Griffiths'86}
Weiren Chou and Robert~B. Griffiths.
\newblock Ground states of one-dimensional systems using effective potentials.
\newblock {\em Phys. Rev. B}, 34:6219--6234, Nov 1986.

\bibitem[CI99]{Contreras'99}
Gonzalo Contreras and Renato Iturriaga.
\newblock {\em Global minimizers of autonomous {L}agrangians}.
\newblock 22$^\circ$ Col\'oquio Brasileiro de Matem\'atica, IMPA, Rio de
  Janeiro. IMPA, 1999.

\bibitem[CIL92]{Crandall'92}
Michael~G. Crandall, Hitoshi Ishii, and Pierre-Louis Lions.
\newblock User’s guide to viscosity solutions of second order partial
  differential equations.
\newblock {\em Bull. Amer. Math. Soc. (N.S.)}, 27:1--67, 1992.

\bibitem[CIPP98]{Contreras'98}
Gonzalo Contreras, Renato Iturriaga, G.~P. Paternain, and M.~Paternain.
\newblock Lagrangian graphs, minimizing measures and {M}a\~n\'e's critical
  values.
\newblock {\em Geom. Funct. Anal.}, 8(5):788--809, 1998.

\bibitem[DFIZ16a]{DFIZ2}
Andrea Davini, Albert Fathi, Renato Iturriaga, and Maxime Zavidovique.
\newblock Convergence of the solutions of the discounted equation: the discrete
  case.
\newblock {\em Math. Z.}, 284(3-4):1021--1034, 2016.

\bibitem[DFIZ16b]{DaviniFathiIturriagaZavidovique2014}
Andrea Davini, Albert Fathi, Renato Iturriaga, and Maxime Zavidovique.
\newblock Convergence of the solutions of the discounted {H}amilton-{J}acobi
  equation: convergence of the discounted solutions.
\newblock {\em Invent. Math.}, 206(1):29--55, 2016.

\bibitem[Fat97a]{Fathi'97_1}
Albert Fathi.
\newblock {\em Comptes Rendus des S\'eances de l'Acad\'emie des Sciences,
  S\'erie I, Math\'ematique}, 324:1043--1046, 1997.

\bibitem[Fat97b]{Fathi'97_2}
Albert Fathi.
\newblock Solutions {KAM} faibles conjugu\'ees et barri\`eres de {P}eierls.
\newblock {\em Comptes Rendus des S\'eances de l'Acad\'emie des Sciences,
  S\'erie I, Math\'ematique}, 325:649--652, 1997.

\bibitem[Fat08]{Fathi'08}
Albert Fathi.
\newblock {\em Weak KAM Theorem in {L}agrangian Dynamics}.
\newblock 2008.

\bibitem[FK38]{FK'38}
Ya.~I. Frenkel and T.~A. Kontorova.
\newblock On the theory of plastic deformation and twinning i, ii, iii.
\newblock {\em Zhurnal Eksperimental'noi i Teoreticheskoi Fiziki}, 8:89--95
  (I), 1340--1349 (II), 1349--1359 (III), 1938.

\bibitem[FR10]{FalconeRorro'10}
Maurizio Falcone and Marco Rorro.
\newblock Optimization techniques for the computation of the effective
  hamiltonian.
\newblock In {\em Recent {A}dvances in {O}ptimization and its {A}pplications in
  {E}ngineering}, pages 225--236. 2010.

\bibitem[GO04]{Gomes'04}
Diogo~A. Gomes and A.D. Oberman.
\newblock Computing the effective {H}amiltonian using a variational approach.
\newblock {\em SIAM J. Control Optim.}, 43(3):792--812, 2004.

\bibitem[Gom05]{Gomes'05}
Diogo~A. Gomes.
\newblock Viscosity solution methods and the discrete {A}ubry-{M}ather problem.
\newblock {\em Discrete Contin. Dyn. Syst.}, 13(1):103--116, 2005.

\bibitem[GT11]{Thieullen'11}
Eduardo Garibaldi and Philippe Thieullen.
\newblock Minimizing orbits in the discrete {A}ubry-{M}ather model.
\newblock {\em Nonlinearity}, 24(2):563--611, 2011.

\bibitem[Ish13]{Ishii2013}
Hitoshi Ishii.
\newblock A short {I}ntroduction to {V}iscosity {S}olutions and the {L}arge
  {T}ime {B}ehavior of {S}olutions of {H}amilton-{J}acobi {E}quations.
\newblock 2074:111--249, 2013.

\bibitem[LPV87]{Lions'87}
Pierre-Louis Lions, George~C. Papanicolaou, and Srinivasa~R.S. Varadhan.
\newblock Homogenization of {H}amilton-{J}acobi equations.
\newblock 1987.
\newblock Preprint.

\bibitem[Mat91]{Mather'91}
John~N. Mather.
\newblock Action minimizing invariant measures for positive definite
  {L}agrangian systems.
\newblock {\em Math. Z.}, 207(2):169--207, 1991.

\bibitem[Mat93]{Mather'93}
John~N. Mather.
\newblock Variational construction of connecting orbits.
\newblock {\em Ann. Inst. Fourier (Grenoble)}, 43(5):1349--1386, 1993.

\bibitem[Mn87]{Manebook}
Ricardo Ma\~n\'e.
\newblock {\em Ergodic theory and differentiable dynamics}, volume~8 of {\em
  Ergebnisse der Mathematik und ihrer Grenzgebiete (3) [Results in Mathematics
  and Related Areas (3)]}.
\newblock Springer-Verlag, Berlin, 1987.
\newblock Translated from the Portuguese by Silvio Levy.

\bibitem[Mn96]{Mane'96}
Ricardo Ma\~n\'e.
\newblock Generic properties and problems of minimizing measures of
  {L}agrangian systems.
\newblock {\em Nonlinearity}, 9:273--310, 1996.

\bibitem[MT14]{Mitake'14}
Hiroyoshi Mitake and Hung~V. Tran.
\newblock Selection problems for a discounted degenerate viscous
  hamilton--jacobi equations.
\newblock 2014.

\bibitem[Nus91]{Nussbaum1991}
Roger~D. Nussbaum.
\newblock Convergence of iterates of a nonlinear operator arising in
  statistical mechanics.
\newblock {\em Nonlinearity}, 4:1223--1240, 1991.

\bibitem[Ror06]{Rorro'06}
Marco Rorro.
\newblock An approximation scheme for the effective {H}amiltonian and
  applications.
\newblock {\em Applied {N}umerical {M}athematics}, 56:1238--1254, 2006.

\bibitem[Zav12]{Zavidovique2012}
Maxime Zavidovique.
\newblock Strict sub-solutions and {M}a\~n\'e potential in discrete weak {KAM}
  theory.
\newblock {\em Comment. Math. Helv.}, 87(1):1--39, 2012.

\end{thebibliography}

\end{document}